\documentclass{birkmult}

\usepackage{amsmath, amsthm, amssymb, mathrsfs}
\usepackage{epsfig, graphicx}

\newtheorem{thm}{Theorem}

\newtheorem{prop}[thm]{Proposition}
\newtheorem{lem}[thm]{Lemma}
\newtheorem*{df}{Definition}
\newtheorem*{acknowledgement}{Acknowledgment}

\newcommand{\T}    {\mathbb{T}}
\newcommand{\D}    {\mathbb{D}}
\newcommand{\A}    {\mathbb{A}}

\newcommand{\C}    {\mathbb{C}}
\newcommand{\N}    {\mathbb{N}}
\newcommand{\Z}    {\mathbb{Z}}
\newcommand{\da}	{\mathfrak{D_f}}
\newcommand{\das}	{\mathfrak{D_f^*}}
\newcommand{\dr}	{\mathfrak{D}}
\newcommand{\Pj}	{\mathcal{P}}
\newcommand{\ho}	{\mathcal{H}}

\newcommand{\gm}	{G}
\newcommand{\gmc}	{M}
\newcommand{\szf}       {S}
\newcommand{\szc}       {D}
\newcommand{\scf}       {Q}

\newcommand{\ed}	{\omega_{[c,d]}}
\newcommand{\ged}	{\omega_{([c,d],\T)}}
\newcommand{\mes}	{\mu}
\newcommand{\map}	{\psi}
\newcommand{\mpc}	{\varphi}

\newcommand{\ct} 	{\mathfrak{f}}
\newcommand{\sr} 	{\mathfrak{w}}
\newcommand{\Arg} 	{\textnormal{Arg}}
\newcommand{\ang} 	{\textnormal{Angle}}

\newcommand{\cp}	{\kappa}
\newcommand{\E}		{\mathscr{E}}
\newcommand{\z}		{\textnormal{z}}
\newcommand{\x}		{{\bf x}}
\newcommand{\supp}	{\textnormal{supp}}

\newcommand{\im} 	{\textnormal{Im}}

\numberwithin{equation}{section}

\begin{document}

\title[Uniform Approximation of Cauchy Integrals]{\LARGE On Uniform Approximation of Rational Perturbations of Cauchy Integrals}

\author[M. Yattselev]{Maxim Yattselev}
\email{myattsel@sophia.inria.fr}
\address{INRIA, Project APICS \\
2004 route des Lucioles --- BP 93 \\
06902 Sophia-Antipolis, France}


\begin{abstract}
Let $[c,d]$ be an interval on the real line and $\mes$ be a measure of the form $d\mes = \dot\mes d\ed$ with $\dot\mes=h\hbar$, where $\hbar(t)=(t-c)^{\alpha_c}(d-t)^{\alpha_d}$, $\alpha_c,\alpha_d\in[0,1/2)$, $h$ is a Dini-continuous non-vanishing function on $[c,d]$ with an argument of bounded variation, and $\ed$ is the normalized arcsine distribution on $[c,d]$. Further, let $p$ and $q$ be two polynomials such that $\deg(p) < \deg(q)$ and $[c,d]\cap \z(q) = \emptyset$, where $\z(q)$ is the set of the zeros of $q$. We show that AAK-type meromorphic as well as diagonal multipoint Pad\'e approximants to
\[
\ct(z) := \int\frac{d\mu(t)}{z-t} + \left(\frac pq\right)(z)
\]
converge locally uniformly to $\ct$ in $\da\cap\D$ and $\da$, respectively, where $\da$ is the domain of analyticity of $\ct$ and $\D$ is the unit disk. In the case of Pad\'e approximants we need to assume that the interpolation scheme is ``nearly'' conjugate-symmetric. A noteworthy feature of this case is that we also allow the density $\dot\mes$ to vanish on $(c,d)$, even though in a strictly controlled manner.
\end{abstract}

\maketitle

\section{Introduction}
\label{sec:intro}

Let $\ct$ be a function of the form
\begin{equation}
\label{eq:ct}
\ct(z) := \int\frac{d\mes(t)}{z-t} + \left(\frac pq\right)(z), \quad \z(q)\cap[c,d]=\emptyset,
\end{equation}
where $[c,d]=\supp(\mes)$ is the support of a complex Borel measure $\mes$, the polynomials $p$ and $q$ are coprime, $\deg(p)<\deg(q)=:m$, and $\z(q)$ is the set of zeros of $q$. Let $\ed$ be the equilibrium distribution for $[c,d]$, which is simply the normalized arcsine distribution. In this paper, we assume that $\mes$ is absolutely continuous with respect to $\ed$ and $\dot\mes$, its Radon-Nikodym derivative ($d\mes=\dot\mes d\ed$), is such that
\begin{subequations}
\label{eq:dotmes}
\begin{gather}
\label{eq:dotmes1} \dot\mes = h\hbar, \\
\label{eq:dotmes2} \quad\dot\mes = h\hbar\hbar_\x,
\end{gather}
\end{subequations}
where $h$ is a non-vanishing Dini-continuous function with argument of bounded variation on $[c,d]$, $\hbar(t) := |t-c|^{\alpha_c}|t-d|^{\alpha_d}$, $\alpha_c,\alpha_d\in[0,1/2)$, $\x\subset(c,d)$ is a finite set of distinct points, and $\hbar_\x(t) := \prod_{x\in\x}|t-x|^{2\alpha_x}$, $\alpha_x\in(0,1/2)$. Under such assumptions on $\ct$, we show locally uniform convergence of $L^p(\T)$-best meromorphic (in this case we assume that $[c,d]\subset(-1,1)$) and certain diagonal multipoint Pad\'e approximants to $\ct$ in
\begin{equation}
\label{eq:da}
\da:= \overline\C\setminus(\supp(\mes)\cup\z(q)),
\end{equation}
the domain of analyticity of $\ct$, where $\overline\C$ is the extended complex plane. It is known \cite{StahlTotik,BS02} that the denominators of both types of approximants satisfy non-Hermitian orthogonality relations with respect to $\mes$ that assume a similar form. This leads to similar integral representations for the error of approximation, which is the reason why we treat them simultaneously.

Generally speaking, meromorphic approximants (MAs) are functions meromorphic in the unit disk that provide an optimal approximation to $\ct$ on the unit circle in the $L^p$-norm when the number of poles is fixed. When considering them, it is customary to assume that $\supp(\mes)\cup \z(q)$ is contained in the unit disk, $\D$. The study of MAs originated from the work of V.M.~Adamyan, D.Z.~Arov, and M.G.~Krein \cite{AAK71}, where the case $p=\infty$ was considered. Nowadays such approximants are often called {\it AAK approximants}. The $L^p$-extensions of the AAK theory were obtained independently by L. Baratchart and F.~Seyfert \cite{BS02} and V.A.~Prokhorov \cite{Pr02}. Meromorphic approximation problems have natural extension to Jordan domains with rectifiable boundary when the approximated function $\ct$ is meromorphic outside of a closed hyperbolic arc of this domain \cite{BMSW06}. However, we shall not consider such a generalization here. 

The AAK theory itself as well as its generalizations is based on the intimate relation between best (locally best) MAs and Hankel operator whose symbol is the approximated function \cite{AAK71,BS02, Pr02}. The study of the asymptotic behavior of MAs is, in fact, equivalent to the study of the asymptotic behavior of the singular vectors and singular numbers of the underlying Hankel operator (see Section~\ref{sec:mer}). Hence, the present work (more specifically, Theorems~\ref{thm:mer1} and~\ref{thm:mer2}) can be considered as an asymptotic analysis of the singular vectors of Hankel operators with symbols of type \eqref{eq:ct}--\eqref{eq:dotmes1}.

Let us briefly account for the existing results on convergence of MAs to functions of type (\ref{eq:ct}). Uniform convergence was obtained in \cite{BStW01} for the case $p=2$ (in this case meromorphic approximants reduce to rational functions) whenever $\mes$ is a positive measure and the rational summand is not present, i.e. $q\equiv1$ and necessarily $p\equiv0$, (such functions $\ct$ are called Markov functions). The general case $p\in[1,\infty]$ was addressed in \cite{BPS01a}, where again only Markov functions were considered and uniform convergence was shown under the assumption that $\mes$ belongs to the {\it Szeg\H{o} class}, i.e. $\log(d\mes(t)/dt)$ is integrable on $[c,d]$. The case of complex measures and non-trivial rational part was taken up in \cite{uBY1}, where convergence in capacity in $\D\setminus\supp(\mes)$ was obtained while $\supp(\mes)$ was assumed to be a regular set with respect to the Dirichlet problem and $\mu$ had to be sufficiently ``thick" on its support and have an argument of bounded variation.

On the other hand, diagonal multipoint Pad\'e approximants (PAs) are rational functions of type $(n,n)$ that interpolate $\ct$ in a system of $2n$ not necessarily distinct nor finite points (interpolation scheme) lying in $\da$ with one additional interpolation condition at infinity. Unlike the meromorphic case, it is pointless to assume that $\supp(\mes)$ and $\z(q)$ lie in $\D$. It is customary to call PA {\it classical} if all the interpolation points lie at infinity. Such approximants were initially studied by A.A.~Markov \cite{Mar95} using the language of continued fractions. Later, A.A.~Gonchar \cite{Gon75a} considered classical PAs to functions of type (\ref{eq:ct}) with nontrivial rational part and positive $\mes$. Locally uniform convergence to $\ct$ in $\da$ was obtained under the condition that $\mes$ belongs to the Szeg\H{o} class. Continuing this work, E.A. Rakhmanov has shown \cite{Rakh77b} that the restriction on $\mes$ to be in the Szeg\H{o} class cannot be relaxed in general, but if all the coefficients of $R$ are real, uniform convergence holds for any positive measure. In the recent paper \cite{GS04} A.A. Gonchar and S.P. Suetin proved that uniform convergence of classical PAs still holds if $\mes$ is a complex measure of the form $d\mes=hd\ed$, where $h$ is a non-vanishing analytic function in some neighborhood of $[c,d]$. Recently, using the operator-theoretic approach, M.S.~Derevyagin and V.A.~Derkach \cite{DerevDerk07} showed that there always exists a subsequence of diagonal PAs that converges locally uniformly to $\ct$ whenever the latter is such that $\mes$ is a positive measure and $p/q$ is real-valued on $\supp(\mes)$ but can have poles there. Finally, we mention a weaker result that holds for a larger class of complex measures. It was shown in \cite[Thm. 2.3]{uBY1} that multipoint Pad\'e approximants corresponding to ``nearly'' conjugate-symmetric interpolation schemes converge in capacity in $\C\setminus\supp(\mes)$ whenever $\supp(\mes)$ is a regular set with respect to the Dirichlet problem and $\mes$ is sufficiently ``thick" on its support and has an argument of bounded variation.

The main results of this paper are presented in Section \ref{sec:mer}, Theorems \ref{thm:mer1}--\ref{thm:mer4}, and Section \ref{sec:pade}, Theorems \ref{thm:pade1} and \ref{thm:pade2}. The conditions imposed on the measure $\mes$ in these theorems come from Theorem \ref{thm:ortho}. The latter is, in fact, a consequence of Theorems 2 and 3 in \cite{uBY3}. In particular, if Theorem \ref{thm:ortho} is established under other assumptions on $\mes$, this would yield Theorems \ref{thm:mer1}--\ref{thm:pade2} for this new class of measures. For instance, all the main results of the present work would hold whenever $\mes$ is of the form
\[
d\mes(t) = h(t)(t-c)^{\alpha_c}(d-t)^{\alpha_d}dt, 
\]
where $h$ is an $m$-times continuously differentiable non-vanishing function on $[c,d]$ with $m$-th derivative being $\varsigma$-H\"older continuous and $\alpha_c,\alpha_d\in(-1,\infty)\cap(-m-\varsigma,m+\varsigma)$ \cite{uBY5}.

\section{Preliminaries and Notation}
\label{sec:notation}

To smoothen the exposition of the material in main Sections \ref{sec:mer} and \ref{sec:pade}, we gather below some necessary prerequisites and notation.

Let $\T_s := \{z:|z|=s\}$, $\T_s^\pm := \T_s\cap\{z:\pm\im(z)\geq0\}$, and $\D_s := \{z:|z|<s\}$, $s>0$, be the circle, the semicircles, and the open disk centered at the origin of radius $s$. For simplicity, we drop the lower index $1$ for the unit circle (semicircles) and the unit disk.

Denote by $H^p$, $p\in[1,\infty]$, the {\it Hardy spaces} of the unit disk consisting of holomorphic functions $f$ in $\D$ such that
\begin{equation}
\label{eq:defNorm}
\begin{array}{lll}
\displaystyle \|f\|_p^p := \sup_{0<s<1}\frac{1}{2\pi}\int_{\T}|f(s\xi)|^p|d\xi|<\infty  & \mbox{ if } & p\in[1,\infty), \smallskip \\
\displaystyle \|f\|_{\infty}:=\sup_{z\in\D}|f(z)|<\infty  & \mbox{ if } & p=\infty.
\end{array}
\end{equation}
It is known \cite[Thm. I.5.3]{Garnett} that a function in $H^p$ is uniquely determined by its trace (non-tangential limit) on the unit circle and that the $L^p$-norm of this trace is equal to the $H^p$-norm of the function, where $L^p$ is the space of $p$-summable functions on $\T$. This way $H^p$ can be regarded as a closed subspace of $L^p$.

In the same vein, we define $\bar H_0^p$, $p\in[1,\infty]$, consisting of holomorphic functions in $\overline\C\setminus\overline\D$ that vanish at infinity and satisfy (\ref{eq:defNorm}) with $1<s<\infty$ and $z\in\overline\C\setminus\overline\D$, respectively. In particular, we have that $L^2=H^2\oplus\bar H^2_0$. Thus, we may define orthogonal projections $\Pj_+:L^2 \to H^2$ (analytic) and $\Pj_-:L^2 \to \bar H^2_0$ (antianalytic). It is easy to see that
\[
\int_\T\frac{h(\xi)}{\xi-z}\frac{d\xi}{2\pi i}  = \left\{
\begin{array}{lll}
\Pj_+(h)(z),  &  z\in\D, \smallskip \\
-\Pj_-(h)(z), & z\in\overline\C\setminus\overline\D, 
\end{array}
\right. \;\;\; h\in L^2.
\]
Recall also the well-known fact \cite[Cor. II.5.8]{Garnett} that any nonzero function in $H^p$ can be uniquely factored as $h=jw$, where 
\[
w(z)=\exp\left\{\frac{1}{2\pi}\int\frac{\xi+z}{\xi-z}\log|h(\xi)||d\xi|\right\}, \quad z\in\D,
\]
belongs to $H^p$ and is called the {\it outer factor} of $h$, while $j$ has modulus 1 a.e. on $\T$ and is called the {\it inner factor} of $h$. The latter may be further decompose as $j=bs$, where $b$ is a Blaschke product, i.e. a function of the form
\[
b(z) = z^k\prod_{z_j\neq0}\frac{-\bar z_j}{|z_j|}\frac{z-z_j}{1-\bar z_jz}
\]
that has the same zeroing as $h$, while $s$ is the {\it singular inner factor}. For simplicity, we often say that a function is outer (resp. inner) if it is equal to its outer (resp. inner) factor.

Continuing with the notation, for any point-set $K$ and any function $f\in H^p$, we denote by $K^*$ and $f^\sigma$ their reflections across $\T$, i.e., $K^* := \{z:1/\bar z\in K\}$ and $f^\sigma(z) := z^{-1}\overline{f(1/\bar z)}$. Clearly then $f\in\bar H_0^p$ and the map $\cdot^\sigma$ is idempotent. Further, for an interval $[c,d]$ we set
\begin{equation}
\label{eq:sr}
\cp:=4/(d-c), \quad \sr(z) := \sqrt{(z-c)(z-d)}, \quad \mbox{and} \quad \dr := \overline\C\setminus[c,d],
\end{equation}
where such a branch of $\sr$ is chosen that $\sr$ is holomorphic in $\dr\setminus\{\infty\}$ and $\sr(z)/z\to1$ as $z\to\infty$. Then
\[
\widetilde\sr(z) := 1/(1/\sr(z))^\sigma = z^2\sr^\sigma(z) = \sqrt{(1-cz)(1-dz)},
\]
is holomorphic in $\dr^*\setminus\{\infty\}$ and $\widetilde\sr(0)>0$. Moreover, the function
\begin{equation}
\label{eq:map}
\map(z) := \frac{2z-(d+c)-2\sr(z)}{d-c}, \quad z\in\dr,
\end{equation}
is the conformal map of $\dr$ onto $\D$ such that $\map(\infty) = 0$ and $\map^\prime(\infty)=\cp>0$. It is also easy to see that $\map$ has well-defined unrestricted boundary values from both side of $[c,d]$ (we assume that the positive side of $[c,d]$ lies on the left when the interval is traversed in the positive direction, i.e.  from $c$ to $d$). Moreover, it holds that
\begin{equation}
\label{eq:psipm}
\map^+\map^- = 1  \quad \mbox{on} \quad [c,d].
\end{equation}

Let now $h$ be a Dini-continuous non-vanishing complex-valued function on $[c,d]$. Recall that Dini-continuity means
\[
\int_{[0,d-c]}\delta^{-1}\max_{|t_1-t_2|\leq\delta}|h(t_1)-h(t_2)|d\delta<\infty.
\]
It can be easily checked (cf. \cite[Sec. 3.3]{uBY3}) that the {\it geometric mean} of $h$, i.e.
\begin{equation}
\label{eq:geommean}
\gm_h := \exp\left\{\int\log h(t)d\ed(t)\right\},
\end{equation}
is independent of the actual choice of the branch of the logarithm and is non-zero. Moreover, the {\it Szeg\H{o} function} of $h$, i.e.
\begin{equation}
\label{eq:szego}
\szf_h(z) := \exp\left\{\frac{\sr(z)}{2}\int\frac{\log h(t)}{z-t}d\ed(t) - \frac12\int\log h(t)d\ed(t)\right\},
\end{equation}
$ z\in\dr$, also does not depend on the choice of the branch (as long as the same branch is taken in both integrals) and is a non-vanishing holomorphic function in $\dr$ that has continuous boundary values from each side of $[c,d]$ and satisfies
\begin{equation}
\label{eq:szegodecomp}
h = \gm_h\szf_h^+\szf_h^- \;\; \mbox{on} \;\; [c,d] \;\; \mbox{and} \;\; S_h(\infty)=1.
\end{equation}
The continuity of the traces of $\szf_h$ is ensured by the Dini-continuity of $h$, essentially because Dini-continuous functions have continuous conjugates \cite[Thm. III.1.3]{Garnett}. In fact, more can be said. Let $a$ be a non-vanishing holomorphic function in $\dr$ that has continuous traces on each side of $[c,d]$ and $a(\infty)=1$. Suppose also that $a^+a^-=c^2$ for some constant $c$. Then the functions $a_i\circ\map:=c/a$ and $a_e\circ(1/\map):=a/c$ are holomorphic in $\D$ and $\overline\C\setminus\overline\D$, respectively, have continuous traces on $\T$, $a_i(0)=c$ and $a_e(\infty)=1/c$. Moreover, it can be readily verified that the traces of $a_i$ and $a_e$ coincide. Thus, $a_i$ and $a_e$ are analytic continuations of each other, from which we deduce by Liouville's theorem that $c=1$ and $a_i\equiv a_e\equiv1$. This simple observation implies the following. Let $h$ be a Dini-continuous function on $[c,d]$ and $\gm$ be some constant. If $\szf$ is a non-vanishing holomorphic function in $\dr$ that assumes value 1 at inifnity, has continuous traces, and is such that $\gm\szf^-\szf^+=h$ then necessarily $\gm=\gm_h$ and $\szf=\szf_h$. It is also true that (\ref{eq:szego}) is well-defined whenever $h$ is a non-negative integrable function with integrable logarithm; like, for example, $\hbar$ and $\hbar_\x$ defined after \eqref{eq:dotmes}.

We also emphasize that the Szeg\H{o} function of a polynomial can be computed in a rather explicit manner as we will now see. Let $v$ be a polynomial with zeros in $\dr$, $\deg(v)\leq k$. Set
\begin{equation}
\label{eq:rn}
r_k(v;z) := \left(\map(z)\right)^{k-\deg(v)}\prod_{e\in\z(v)}\left( \frac{\map(z)-\map(e)}{1-\map(z)\map(e)}\right)^{m(e)}, \quad z\in\dr,
\end{equation}
where $\z(v)$ is the set of zeros of $v$ and $m(e)$ is the multiplicity of $e\in\z(v)$. Then $r_k(v;\cdot)$ is a holomorphic function in $\dr$ with a zero of multiplicity $m(e)$ at each $e\in\z(v)$ and a (possible) zero of multiplicity $k-\deg(v)$ at infinity. Moreover, it has unrestricted continuous boundary values from both sides of $[c,d]$ such that
\begin{equation}
\label{eq:rkpm}
r_k^+(v;\cdot)r_k^-(v;\cdot) = 1
\end{equation}
by \eqref{eq:psipm}. Then since $\szf_v$ is the unique function of Szeg\H{o} type such that $\szf^+_v\szf^-_v$ is equal to a constant multiple of $v$, it holds that
\begin{equation}
\label{eq:szegopoly}
\szf_v^2 = \frac{1}{\gm_v}\frac{v\map^k}{r_k(v;\cdot)}.
\end{equation}

In some cases it will be important to consider the ratio of the boundary values of Szeg\H{o} functions rather then their product. Hence, we introduce
\begin{equation}
\scf^\pm_h(t) := \szf^\pm_h(t)/\szf^\mp_h(t), \quad t\in[c,d].
\end{equation}
When $h$ is non-vanishing Dini-continuous function, $\scf_h^\pm$ are continuous on $[c,d]$ and assume the value 1 at the endpoints. Finalizing the discussion on Szeg\H{o} functions, let us state two of their properties that we shall use implicitly on several occasions and the reader will have no difficulty to verify. The first one is the multiplicativity property, i.e. $\szf_{h_1h_2}=\szf_{h_1}\szf_{h_2}$, and the second one is the convergence property which says that $\szf_{h_n}=[1+o(1)]\szf_h$ uniformly in $\overline\C$, i.e. including the boundary values, whenever $h_n=[1+o(1)]h$ uniformly on $[c,d]$. For more information on Szeg\H{o} functions of complex $h$, the reader may consult \cite[Sec. 3.3]{uBY3}.

Next, we denote by $\A_{s_1,s_2} := \{z:~s_1<|z|<s_2\}$, $0<s_1<s_2$, and $\A_s := \A_{s,1/s}$, $s<1$, the annuli centered at the origin and by $\mpc$ the conformal map from $\dr\cap\dr^*$ onto $\A_\rho$, $\mpc(1)=1$. Recall that annuli $\A_s$ are not conformally equivalent for different $s$ and therefore $\rho=\rho([c,d])$ is uniquely determined by $[c,d]$. From the potential-theoretic point of view $\rho$ can be expressed as
\begin{equation}
\label{eq:rho}
\rho = \exp\left\{-\frac{1}{\textnormal{cap}([c,d],\T)}\right\},
\end{equation}
where $\textnormal{cap}([a,b],\T)$ is the {\it capacity of the condenser} $([c,d],\T)$. The map $\mpc$ is given by \cite[Thm. VIII.6.1]{SaffTotik}
\begin{equation}
\label{eq:maps}
\mpc(z) = \exp\left\{\mathcal{T}^2\int_1^z\frac{dt}{(\sr\widetilde\sr)(t)}\right\}
\end{equation}
with integration taken along any path in $\dr\cap\dr^*$, where
\begin{equation}
\label{eq:tau}
\mathcal{T}^{-2} := \frac2\pi\int_{[0,1]}\frac{dx}{\sqrt{(1-x^2)((1-cd)^2-(d-c)^2x^2)}}.
\end{equation}
Moreover, it holds that $\mpc(\bar z)=\overline{\mpc(z)}$ and $\mpc(1/z)=1/\mpc(z)$, $z\in\dr\cap\dr^*$. Thus, $\mpc(\T)=\T$ and $\mpc(\D\setminus [c,d])=\A_{\rho,1}$. Further, it is not hard to check that $\mpc$ extends continuously on each side of $[c,d]$ (resp. $[c,d]^*$) and $\mpc^\pm([c,d])=\T_\rho^\pm$ (\mbox{resp.} $\mpc^\pm([c,d]^*)=\T_{1/\rho}^\pm)$. Finally, the \emph{Green equilibrium distribution} (which is a probability measure on $[c,d]$ \cite[Sec. II.5]{SaffTotik}) for the condenser $\D\setminus[c,d]$, as well as for the condenser $\dr\cap\dr^*$, is given by
\begin{equation}
\label{eq:ged}
d\ged(t) = \frac{\mathcal{T}^2dt}{\pi|(\sr^+\widetilde\sr)(t)|} = \frac{|d\mpc^+(t)|}{\pi\rho},
\end{equation}
where the normalization follows from \eqref{eq:tau} and the second equality holds by differentiating \eqref{eq:maps} and taking boundary values.

\section{Meromorphic Approximation}
\label{sec:mer}

The meromorphic approximants (MAs) that we deal with are defined as follows. For $p\in[1,\infty]$ and $n\in\N$, the {\it class of meromorphic functions of degree} $n$ in $L^p$ is
\begin{equation}
\label{eq:merfuns}
H^p_n := H^pB_n^{-1},
\end{equation}
where $B_n$ is the set of Blaschke products of degree at most $n$ (with at most $n$ zeros). By the celebrated theorem of Adamyan, Arov, and Krein \cite{AAK71}  (see also \cite[Ch. 4]{Peller}) and its generalizations \cite{BS02,Pr02} it is known that for any fixed $n\in\N$ and $p\in[1,\infty]$ and given $f\in L^p$ there exists a meromorphic function $g_n$ such that
\begin{equation}
\label{eq:merAppr}
\|f-g_n\|_p=\inf_{g\in H_n^p}\|f-g\|_p.
\end{equation}
Moreover, $g_n$ is unique when $p\in[1,\infty)$, but in the case $p=\infty$ it is necessary to assume $f\in H^\infty+C(\T)$ to ensure uniqueness of $g_n$, where $C(\T)$ is the space of continuous functions on the unit circle. Obviously, when $\supp(\mes)\subset\D$ and $q$ has no zeros on $\T$ the function $\ct$ defined in (\ref{eq:ct}) complies with these requirements for any $p\in[1,\infty]$. When $p<2$, no functional representation for the error is known to satisfy orthogonality relations \cite{BS02}. This is the reason why in what follows we shall restrict to the case $p\in[2,\infty]$.

Due to similar functional decomposition and their appearances in the computations\footnote{It is most likely that a numerical search ends up with stable critical points, that is locally best MAs, rather than just best MAs.}, we consider not only best MAs but more generally {\it critical point} of meromorphic approximation problem (\ref{eq:merAppr}). Although their definition is rather technical (see below), critical points are just those $g_n=h_n/b_n\in H_n^p$ (see (\ref{eq:merfuns})) for which the derivative of $\|f-g_n\|_p$ with respect to $b_n\in B_n$ and $h_n\in H^p$ does vanish \cite{BS02}. By definition, a function $g_n$ is a critical point of order $n$ in meromorphic approximation problem (\ref{eq:merAppr}) if and only if it assumes the form
\begin{equation}
\label{eq:errFunDec}
g_n=f-\frac{\ho_f(v_n)}{v_n} = \frac{\Pj_+(fv_n)}{v_n},
\end{equation}
where $\ho_f$ is {\it Hankel} operator with a symbol $f$, i.e.
\[
\ho_f:H^{p^\prime}\to\bar H^2_0, \;\;\; \ho_f(h):=\Pj_-(fh), \;\;\; 1/p + 1/p^\prime = 1/2,
\]
and $v_n\in H^{p^\prime}$ is of unit norm (a Blaschke product if $p=2$), its inner factor is a Blaschke product of exact degree $n$, and is such that
\[
\begin{array}{ll}
\ho_f^*\ho_f(v_n)=\sigma_n^2\Pj_+\left(|v_n|^{p^\prime-2}v_n\right) & ~\mbox{if } ~p>2, \\
\ho_f^*\ho_f(v_n)=\Pj_+\left(|\ho_f(v_n)|^2v_n\right) & ~\mbox{if } ~p=2,
\end{array}
\]
with $\ho_f^*$ being the adjoint operator. A function $v_n$ is called a {\it singular vector} associated to a critical point $g_n$ and 
\begin{equation}
\label{eq:CriticalValue}
\sigma_n:=\|f-g_n\|_p=\|\ho_f(v_n)\|_p, \;\;\; p\in[2,\infty],
\end{equation}
is called the {\it critical value} associated to $g_n$. In the case when $g_n$ is a best MA to $f$ it also holds that $\sigma_n$ is the $n$-th singular number of $\ho_f$, i.e.
\[
\sigma_n = \sigma_n(\ho_f):=\inf\left\{\|\ho_f-\Gamma\|:~\Gamma:H^{p^\prime}\to\bar H_0^2 ~\mbox{linear operator of rank}~\leq n\right\};
\]
when $p=2$ it is assumed in addition that $\Gamma$ is weak$^*$ continuous. Hereafter, we use the following notation for the inner-outer decomposition of singular vectors:
\begin{equation}
\label{eq:sinVecDec}
v_n= b_nw_n, \quad b_n=q_n/\widetilde q_n, \quad w_n(0)>0, \quad \widetilde q_n(z) = z^n\overline{q_n(1/\bar z)},
\end{equation}
where $w_n$ is an outer factor, $q_n$ is a monic polynomial of exact degree $n$, and $\widetilde q_n$ is the {\it reciprocal polynomial} of $q_n$. To uniformize the notation, we simply set $w_n\equiv1$ when $p=2$.

A critical point of order $n$ may have less than $n$ poles, even though we insisted in the definition that $v_n$ has exactly $n$ zeros. Cancellation may occur due to zeros of $\Pj_+(fv_n)$. When this is not the case, we shall call $g_n$ an {\it irreducible} critical point. It is worth mentioning that when $p\in[2,\infty)$ a best MA is not necessarily unique, but has exactly $n$ poles. Thus, all best MAs are irreducible critical points. To the contrary, if $p=\infty$, best MA is unique and is the only critical point of order $n$, but may have less then $n$ poles. However, there always exists a subsequence of natural numbers for which best AAK approximants are irreducible. Since the behavior of the poles of MAs is entirely characterized by this subsequence, hereafter we say ``a sequence of irreducible critical points'' to mean if $p=\infty$ that we pass to a subsequence if needed.

Now, we are ready to state the first theorem of this section.

\begin{thm}
\label{thm:mer1}
Let $\ct$ be given by \eqref{eq:ct} and \eqref{eq:dotmes1}. Further, let $\{g_n\}$ be a sequence of irreducible critical points of the meromorphic approximation problem to $\ct$, $p\in(2,\infty]$. Then the outer factors $w_n$ in \eqref{eq:sinVecDec} are such that
\begin{equation}
\label{eq:mer1}
w_n^{p^\prime/2} = \frac{\mathcal{T}+o(1)}{\widetilde\sr} + \frac{l_n}{\widetilde q}, \quad \widetilde q(z)=z^m\overline{q(1/\bar z)},
\end{equation}
where $o(1)$ holds locally uniformly in $\dr^*$, $\mathcal{T}$  was defined in \eqref{eq:tau}, and the polynomials $l_n$, $\deg(l_n)<m$, converge to zero and are coprime with $\widetilde q$.
\end{thm}

This theorem is a strengthening of Lemma 3.4 in \cite{uBY1} that asserts, under much milder assumptions on $\mes$, that $\left\{w_n^{p^\prime/2}\right\}$ is a normal family in $\das$ and any limit point of $\{w_n\}$ in $\D$ is zero free.

For simplicity, set $w:=\left(\mathcal{T}/\widetilde\sr\right)^{2/p^\prime}$ for each $p>2$. Then Theorem~\ref{thm:mer1} yields that $w_n \to w$, uniformly in some neighborhood of $\overline\D$ (it follows from the proof of Theorem \ref{thm:mer1} and can be seen from asymptotic formula \eqref{eq:mer1} that the outer factors $w_n$ can be extended to holomorphic functions in any simply connected neighborhood of $\overline\D$ contained in $\das$). Now, we are ready to describe the asymptotic behavior of irreducible critical points.

\begin{thm}
\label{thm:mer2}
Let $\ct$ and $\{g_n\}$ be as in Theorem \ref{thm:mer1}. Then the Blaschke products $b_n$ in \eqref{eq:sinVecDec} are such that
\begin{equation}
\label{eq:mer2}
b_n = [1+o(1)]b \mpc^{n-m}\szc_n^{-1} \quad \mbox{locally uniformly in} \quad \da\cap\das,
\end{equation}
where $b:=q/\widetilde q$, $m=\deg(q)$, $\{\szc_n\}$ is a normal family of non-vanishing functions in  $\dr\cap\dr^*$, such that  $|\szc^\pm_n|$ are uniformly bounded above and away from zero on $[c,d]$ and $[c,d]^*$. Moreover, the following error estimates take place
\begin{equation}
\label{eq:mer3}
\sigma_n = \|\ct-g_n\|_p =  \left[2\gmc\mathcal{T}^{-1}+o(1)\right]\rho^{2(n-m)},
\end{equation}
where $\sigma_n$ is the critical value associated to $g_n$ {\it via} \eqref{eq:CriticalValue}, and
\begin{equation}
\label{eq:mer4}
(\ct-g_n) = \frac{2\gmc+o(1)}{w\sr} \left(\frac{\rho}{\mpc}\right)^{2(n-m)} \frac{\szc_n^2}{b^2}
\end{equation}
uniformly on compact subsets of $\da\cap\overline\D$, where
\begin{equation}
\label{eq:m}
\gmc := \exp\left\{ \int \log\left|(b^2w\dot\mes)(t)\right| d\ged(t) \right\},
\end{equation}
is the geometric mean of $|b^2w\dot\mes|$ with respect to the condenser $\dr\cap\dr^*$.
\end{thm}

It is worth mentioning that the functions $\szc_n$ are, in fact, {\it Szeg\H{o} functions for the condenser} $\dr\cap\dr^*$ that first were introduced in \cite[Def. 2.38]{BStW01} for the case of a positive measure $\mes$. In such a situation the Szeg\H{o} function for a condenser has an integral representation that is no longer valid for complex measures. Moreover, the normalization in the complex case is more intricate (see Proposition \ref{pr:2}). Nevertheless, it still holds that the functions $\szc_n$ have zero winding number on any curve separating $[c,d]$ from $[c,d]^*$, $\szc_n(z)\overline{\szc_n(1/\bar z)}=1$, $z\in\dr\cap\dr^*$, and $\gmc|\szc_n^+\szc_n^-|=|b^2w\dot\mes|$ on $[c,d]$.

We remind the reader that the case $p=2$ has a couple of special traits. First, best MA $g_n$ specializes to a rational function. Indeed, $g_n$ can be written as a sum $h_n+p_{n-1}/q_n$, where $h_n\in H^2$, $\deg(p_{n-1})<\deg(q_n) = n$. As $L^2=H^2\oplus\bar H^2_0$ and $\ct\in\bar H^2_0$, we have that
\[
\|\ct-g\|_2^2 = \|h\|_2^2 + \|\ct-p_{n-1}/q_n\|_2^2.
\]
Hence, to achieve the minimum of the left-hand side of the equality above, one necessarily should take $h\equiv0$. This is the reason why we referred on some occasions to the meromorphic approximation problem with $p=2$ as to the rational approximation problem. Second, the outer factors $w_n$ in \eqref{eq:sinVecDec} are not present, or better assumed to be identically 1. The latter allows us to consider a slightly larger class of measures, namely those given by \eqref{eq:dotmes2}.

\begin{thm}
\label{thm:mer3}
Let $\ct$ be given by \eqref{eq:ct} and \eqref{eq:dotmes2}, where $\sin(\alpha_x\pi)\in(0,\Psi_x)$,
\[
\Psi_x := \min_{\pm}\{|\scf^\pm_{q^2h}(x)|\}\exp\left\{-\frac{4s_1[V_h+2m\pi]}{1-s_0}\right\}, \quad x\in\x,
\]
$V_h$ is the total variation of the argument of $h$ on $[c,d]$, $s_0:=\max_\T|\map|$, and $s_1:=\max_\T|\map^\prime|$. Let further $\{g_n\}$ be a sequence of irreducible critical points of the meromorphic approximation problem with $p=2$ to $\ct$. Then \eqref{eq:mer1}, \eqref{eq:mer2}, and \eqref{eq:mer3} hold with $w\equiv1$.
\end{thm}

It follows from (\ref{eq:mer2}) that each $b_n$ has exactly $m$ zeros approaching the zeros of $b$. In fact, it is possible to say more.

\begin{thm}
\label{thm:mer4}
For each $\eta\in\z(q)$ and all $n$ large enough, there exists an arrangement of $\eta_{1,n},\ldots\eta_{m(\eta),n}$, the zeros of $b_n$ approaching $\eta$, such that
\begin{equation}
\label{eq:mer5}
\eta_{k,n} = \eta + A_{k,n}^\eta \left(\frac{\rho}{\mpc(\eta)}\right)^{2(n-m)/m(\eta)} \exp\left\{\frac{2\pi ki}{m(\eta)}\right\}, \;\;\; k=1,\ldots,m(\eta),
\end{equation}
where the sequences $\{\max_k|A_{k,n}^\eta|\}$ and $\{\max_k|1/A_{k,n}^\eta|\}$ are bounded above.
\end{thm}

This theorem essentially says that each pole $\eta$ of $\ct$ attracts exactly $m(\eta)$ poles of $g_n$, the latter converge geometrically fast and are asymptotically distributed as the roots of unity of order $m(\eta)$. The proof of this theorem is an adaptation of the technique developed in \cite{GS04} for classical Pad\'e approximants to Cauchy transforms of analytic densities. As one can see from the next section, similar results hold not only for classical but more generally for multipoint Pad\'e approximants to Cauchy transforms of less regular measures.

\section{Multipoint Pad\'e Approximation}
\label{sec:pade}

Let $\ct$ be given by (\ref{eq:ct}). Classically, diagonal (multipoint) Pad\'e approximants to $\ct$ are rational functions of type $(n,n)$ that interpolate $\ct$ at a prescribed system of $2n+1$ points. However, when the approximated function is of the form \eqref{eq:ct}, it is customary to place at least one interpolation condition at infinity. More precisely, let $\E=\{E_n\}$ be a sequence of sets each consisting of $2n$ not necessarily distinct nor finite points in $\da$ ({\it interpolation scheme}), and let $v_n$ be the monic polynomial with zeros at the finite points of $E_n$.

\begin{df}[Pad\'e Approximants]
Given $\ct$ of type (\ref{eq:ct}) and an interpolation scheme $\E$, the n-th diagonal Pad\'e approximant to $\ct$ associated with $\E$ is the unique rational function $\Pi_n=p_n/q_n$ satisfying:
\begin{itemize}
\item $\deg p_n\leq n$, $\deg q_n\leq n$, and $q_n\not\equiv0$;
\item $\left(q_n(z)\ct(z)-p_n(z)\right)/v_n(z)$ is analytic in $\da$;
\item $\left(q_n(z)\ct(z)-p_n(z)\right)/v_n(z)=O\left(1/z^{n+1}\right)$ as $z\to\infty$.
\end{itemize}
\end{df}

A Pad\'e approximant always exists since the conditions for $p_n$ and $q_n$ amount to solving a system of $2n+1$ homogeneous linear equations with $2n+2$ unknown coefficients, no solution of which can be such that $q_n\equiv0$ (we may thus assume that $q_n$ is monic); note the required interpolation at infinity is entailed by the last condition and therefore $\Pi_n$ is, in fact, of type $(n-1,n)$.

By the very definition, the behavior of $\Pi_n$ depends on the choice of the interpolation scheme. We define {\it the support} of $\E=\{E_n\}$ as $\supp(\E):=\cap_{n\in\N}\overline{\cup_{k\geq n}E_k}$. Hereafter, the counting measure of a finite set is a probability measure that has equal mass at each point counting multiplicities and the weak$^*$  topology is understood with respect to the duality between complex measures and continuous functions with compact support in $\overline\C$.

\begin{df}[Admissibility]
An interpolation scheme $\E$ is called admissible if
\begin{itemize}
\item there exist rearrangements $\Delta_n$ of $E_n$ such that the sums $\sum_{e\in E_n}| \map(\bar e)-\map(\Delta_n(e))|$ are uniformly bounded when $n\to\infty$;
\item $\supp(\E)\subset\da$ and the probability counting measures of points in $E_n$ converge weak$^*$ to some Borel measure with finite Green energy\footnote{For information on the notions of potential theory we refer the reader to the monographs \cite{Ransford,SaffTotik}.}  relative to $\dr$.
\end{itemize}
\end{df}

Then the following result holds.

\begin{thm}
\label{thm:pade1}
Let $\{\Pi_n\}$ be a sequence of diagonal Pad\'e approximants associated with an admissible interpolation scheme $\E=\{E_n\}$ to $\ct$ given by \eqref{eq:ct} and \eqref{eq:dotmes2} with 
\begin{equation}
\label{eq:upsilon}
\alpha_x\pi\in(0,\arcsin\Upsilon_x), \quad \Upsilon_x := \liminf_{n\to\infty} \min_{\pm} \left\{ |(r_n\scf_h)^\pm(x)| \right\},
\end{equation}
for any $x\in\x$, where $r_n:=r_{2n}(v_n;\cdot)$. Then
\begin{equation}
\label{eq:pade1}
(\ct-\Pi_n)\sr = [2\gm_{\dot\mes}+o(1)]~\szf_{\dot\mes}^2~\frac{r_n}{r^2}
\end{equation}
locally uniformly in $\da$, where $r:=r_m(q;\cdot)$. When $(p/q)\equiv0$ it is not necessary to assume boundedness of the variation of the argument of $h$.
\end{thm}

We would like to point out that $\Upsilon_x$ is, in fact, continuous function of $x$ on $[c,d]$ such that $\Upsilon_c=\Upsilon_d=1$ and $\min_{x\in[c,d]}\Upsilon_x>0$. The latter is true since the  functions $\scf_h^\pm$ are non-vanishing and continuous on $[c,d]$. Moreover, it will be shown in the proof of Theorem \ref{thm:ortho} that the admissibility of $\E$ implies uniform boundedness of $|r_n^\pm|$ and hence their uniform boundedness away from zero by (\ref{eq:rkpm}). It is also easy to check that when the sets $E_n$ are conjugate-symmetric and $h$ is a positive function, it holds that $\Upsilon_x\equiv1$.

Concerning the behavior of $\Pi_n$ near polar singularities of $\ct$, i.e. near $\z(q)$, the following theorem asserts the same ``roots of unity'' behavior as in Theorem \ref{thm:mer4} and is a generalization of \cite[Thm. 3]{GS04} for the case of multipoint Pad\'e approximants and less regular measures.

\begin{thm}
\label{thm:pade2}
Under the conditions of Theorem \ref{thm:pade1} let $q_n$ be the denominators of $\Pi_n$. Then
\begin{equation}
\label{eq:pade2}
q_n = u_{n-m}q_{n,m} \quad \mbox{and} \quad q_{n,m}=(1+o(1))q,
\end{equation}
where $\deg(u_{n-m})=n-m$, $\deg(q_{n,m})=m$, the polynomials $u_{n-m}$ have no zeros on any closed set in $\dr$ for all $n$ large enough, and $o(1)$ holds locally uniformly in $\overline\C\setminus \z(q)$. Moreover, for each $\eta\in\z(q)$ with multiplicity $m(\eta)$ and all $n$ large enough there exists an arrangement of $\eta_{1,n}, \ldots, \eta_{m(\eta),n} \in\z(q_{n,m})$ such that
\begin{equation}
\label{eq:pade3}
\eta_{k,n} = \eta + A_{k,n}^\eta\left[r_n(\eta)\right]^{1/m(\eta)}\exp\left\{\frac{2\pi ki}{m(\eta)}\right\}, \quad k=1,\ldots,m(\eta),
\end{equation}
where the sequence $\{\max_k|A_{k,n}^\eta|\}$ is bounded above.
\end{thm}

\section{Non-Hermitian Orthogonal Polynomials}
\label{sec:ortho}

In this section we describe the asymptotic behavior of non-Hermitian orthogonal polynomials with varying weights on $[c,d]$. In what follows, we assume that $\{\nu_n\}$ is a sequence of complex measures on $[c,d]$ such that
\[
d\nu_n = \dot\nu_nd\ed, \quad \dot\nu_n = hh_n\hbar\hbar_\x/v_{n+m}, \quad m\in\Z_+,
\]
where $h$ is a non-vanishing Dini-continuous function on $[c,d]$, $\{h_n\}$ is a normal family of non-vanishing functions in some neighborhood of $[c,d]$ none of which limit points can vanish in this neighborhood, $v_n$, $\deg(v_n)\leq2n$, are monic polynomials with zeros at finite points of an admissible interpolation scheme, $\hbar$ and $\hbar_\x$ are as in the introduction with 
\begin{equation}
\label{eq:Upsilon}
\alpha_x\pi\in(0,\arcsin\Upsilon_x), \quad \Upsilon_x := \liminf_{n\to\infty} \min_{\pm} \left\{ |(r_{n+m}\scf_{hh_n})^\pm(x)| \right\},
\end{equation}
for each $x\in\x$, where $r_k:=r_{2k}(v_k;\cdot)$. Observe that we do not require $h$ to have argument of bounded variation. Then the following theorem holds.

\begin{thm}
\label{thm:ortho}
Let $\{\nu_n\}$ be as described and $\{u_n\}$ be a sequence of polynomials satisfying
\[
\int t^ju_n(t)d\nu_n(t) = 0, \quad j=0,\ldots,n-1,
\]
and $\{R_n\}$ be the sequence of corresponding functions of the second kind, i.e.,
\begin{equation}
\label{eq:secondkind}
R_n(z) := \int\frac{u_n(t)}{z-t}d\nu_n(t) = \frac{1}{u_n(z)}\int\frac{u_n^2(t)}{z-t}d\nu_n(t).
\end{equation}
Then, for all $n$ large enough, the polynomials $u_n$ have exact degree $n$ and therefore can be normalized to be monic. Under such a normalization it holds that
\begin{equation}
\label{eq:nhop}
\left\{
\begin{array}{lll}
u_n  &=& [1+o(1)]/\szf_n \\
R_n\sr &=& [1+o(1)]\gamma_n\szf_n
\end{array}
\right. \;\;\; \mbox{locally uniformly in} \;\; \dr,
\end{equation}
where $\szf_n := \szf_{\dot\nu_n}(\cp\map)^n$, $\gamma_n:=2\cp^{-2n}\gm_{\dot\nu_n}$, and $\cp$ and $\sr$ were defined in \eqref{eq:sr}.
\end{thm}
\begin{proof} This theorem is an adaptation of \cite[Thm. 3]{uBY3}. To see this we need several observations. Firstly, the orthogonality relations in \cite[Thm. 3]{uBY3} are considered on Jordan arcs connecting $-1$ and $1$, of which the interval $[-1,1]$ is a particular case. The current setting can be easily deduced by applying a linear transformation $l(x)=[(d-c)x+d+c]/2$.

Secondly, $\{h_n\}$ is taken in \cite[Thm. 3]{uBY3} to be a family of Dini-continuous non-vanishing functions on $[c,d]$ such that any sequence in this family contains a uniformly convergent subsequence to a non-vanishing function and the moduli of continuity of $h_n$ are bounded by the same fixed modulus of continuity. Clearly, the normality of $\{h_n\}$ yields that all these restrictions are met in the present case.

Thirdly, only the case $m=0$ is considered in \cite[Thm. 3]{uBY3}. However, the general case we are dealing with is no different. Indeed, choose $2m$ zeros of each polynomial $v_{n+m}$ that converge to some fixed point in $\dr$ (recall that the counting measures of zeros of $v_n$ converge in the weak$^*$ sense) and pay the polynomial, say $p_{2m,n}$, vanishing at these points to $h_n$. Then $\{h_n/p_{2m,n}\}$ is again, a normal family of holomorphic functions with the required properties and the new polynomial factor $v_{n+m}/p_{2m}$ of $\dot\nu_n$ has degree no greater than $2n$.

Finally, in order to appeal to \cite[Thm. 3]{uBY3}, we need to show that the functions $r_n=r_{2n}(v_n;\cdot)$ are such that $r_n=o(1)$ locally uniformly in $\dr$, $|r^\pm_n|=O(1)$ on $[c,d]$, and the moduli of continuity of $|r_n^\pm\circ\map^{-1}|$ are bounded by the same fixed modulus of continuity\footnote{Observe that $\map^{-1}$ is just the Joukovski transformation.}. To do so, consider
\[
\hat r_n (z) := r_n(\map^{-1}(z)) = \prod_{c\in \hat E_n}\frac{z-c}{1-cz}, \quad \hat E_n:=\{\map(e):~e\in E_n\}, \quad z\in\D.
\]
Observe that the sets $E_n$ lie at fixed positive distance from $[c,d]$ by the assumption $\supp(\E)\subset\da$ and therefore there exists $s_1<1$ such that $\hat E_n\subset\D_{s_1}$ for all $n$. Thus, the Blaschke products
\[
\hat b_n(z) := \prod_{c\in \hat E_n}\frac{z-c}{1-\bar cz} = o(1)
\]
locally uniformly in $\D$ \cite[Thm. 2.2.1]{Garnett}. Further,
\[
|\map(z_1)-\map(z_2)| \leq\max_{z\in K}|\map^\prime(z)| |z_1-z_2|, \quad z_1,z_2\in K := \map^{-1}(\overline\D_{s_1}),
\]
and we get from the admissibility of $\{E_n\}$ that
\[
\sum_{c\in\hat E_n}|\bar c-\hat\Delta_n(c)| \leq s_2, \quad \hat \Delta_n(c) := \map\left(\Delta_n\left(\map^{-1}(c)\right)\right),
\]
for all $n$ and some positive constant $s_2$. Consider now the functions
\[
(\hat r_n/\hat b_n)(z) = \prod_{c\in \hat E_n}\frac{1-\bar cz}{1-cz} = \prod_{c\in \hat E_n}\frac{1-\bar cz}{1-\hat\Delta_n(c)z}, \quad z\in\overline\D.
\]
Clearly, this is a sequence of outer functions in $\D$. Moreover,
\[
\log|(\hat r_n/\hat b_n)(\tau)| = \sum_{c\in\hat E_n}\log\left|1+\frac{(\hat\Delta_n(c)-\bar c)\tau}{1-\hat\Delta_n(c)\tau}\right| \leq \sum_{c\in\hat E_n}\frac{\left|\hat\Delta_n(c)-\bar c\right|}{1-|\hat\Delta_n(c)|} \leq \frac{s_2}{1-s_1}
\]
for $\tau\in\T$. Thus, we have that
\[
|\hat r_n| \leq  s_3|\hat b_n| = o(1), \quad s_3:=\exp\{s_2/(1-s_1)\},
\]
locally uniformly in $\D$ and
\begin{equation}
\label{eq:boundrnhat}
|\hat r_n| = |\hat r_n/\hat b_n| \leq s_3 \quad \mbox{on} \quad \T.
\end{equation}
Therefore, the corresponding properties of $r_n$ and $|r_n^\pm|$ follow.

Next, we show that $|r_n^\pm\circ\map^{-1}|$ have moduli of continuity majorized by the same function. As $|\hat b_n|\equiv1$ on $\T$, it is enough to consider $|\hat r_n/\hat b_n|$. Let $\tau_1,\tau_2\in\T$. Then
\[
\log\left|\frac{(\hat r_n/\hat b_n)(\tau_2)}{(\hat r_n/\hat b_n)(\tau_1)}\right| = \sum_{c\in\hat E_n}\log \left|1 + \frac{(\tau_1-\tau_2)(\bar c-\hat\Delta_n(c))}{(1-\bar c\tau_1)(1-\hat\Delta_n(c)\tau_2)}\right| \leq \frac{s_2|\tau_1-\tau_2|}{(1-s_1)^2}.
\]
Therefore, we have with $s_4:=s_2/(1-s_1)^2$ that
\[
\exp\{-s_4|\tau_1-\tau_2|\} \leq \left|\frac{(\hat r_n/\hat b_n)(\bar\tau_1)}{(\hat r_n/\hat b_n)(\bar\tau_2)}\right| = \left|\frac{(\hat r_n/\hat b_n)(\tau_2)}{(\hat r_n/\hat b_n)(\tau_1)}\right| \leq \exp\{s_4|\tau_1-\tau_2|\}.
\]
Moreover, denoting by $\Arg(z)\in(-\pi,\pi]$ the principal argument of $z\neq0$ and using
\[
|\Arg(1+z)|\leq \arcsin|z| \leq \pi|z|/2, \quad |z|<1,
\]
we get that
\[
\left|\Arg\left(\frac{(\hat r_n/\hat b_n)(\tau_2)}{(\hat r_n/\hat b_n)(\tau_1)}\right)\right| \leq \sum_{c\in\hat E_n}\left|\Arg \left(1 + \frac{(\tau_1-\tau_2)(\bar c-\hat\Delta_n(c))}{(1-\bar c\tau_1)(1-\hat\Delta_n(c)\tau_2)}\right)\right| \leq \frac{s_4|\tau_1-\tau_2|}{2/\pi} \nonumber
\]
for $s_4|\tau_1-\tau_2|\leq1$. Hence, for such $\tau_1$ and $\tau_2$ we obtain that
\[
\left|(\hat r_n/\hat b_n)(\tau_1)-(\hat r_n/\hat b_n)(\tau_2)\right| \leq s_3\left|1 - \frac{(\hat r_n/\hat b_n)(\tau_2)}{(\hat r_n/\hat b_n)(\tau_1)}\right| \leq s_5|\tau_1-\tau_2|
\]
for some absolute constant $s_5$. This finishes the proof of the theorem, granted \cite[Thm.~3]{uBY3}.
\end{proof}

\section{Proofs of Theorems \ref{thm:mer1}--\ref{thm:mer4}}
\label{sec:proofs2}

We start by providing several auxiliary results.

\begin{lem}
\label{lem:auxS}
Under the conditions of either Theorem \ref{thm:mer2} or Theorem \ref{thm:mer3} it holds that
\begin{equation}
\label{eq:Qnm}
q_n = u_{n-m}q_{n,m}, \quad q_{n,m}= (1+o(1))q,
\end{equation}
locally uniformly in $\overline\C\setminus\z(q)$, where $\deg(u_{n-m})=n-m$ and $\deg(q_{n,m})=m$. Moreover, the zeros of the polynomials $\widetilde q_n^2$ form an admissible interpolation scheme.
\end{lem}
\begin{proof} It follows from \cite[Thm. 2.4]{uBY1} that if $\supp(\mes)$ is a regular set with respect to the Dirichlet problem, $\dot\mes$ has an argument of bounded variation on $\supp(\mes)$, and $|\mes|([x-\delta,x+\delta])\geq l\delta^L$ for all $x\in[c,d]$ and some fixed positive constants $l$ and $L$, then in any neighborhood of $\eta\in\z(q)$ the polynomials $q_n$ have at least $m(\eta)$ zeros for all $n$ large enough (in fact, no more then $m(\eta)$ plus an absolute constant depending only on $\ct$), which is indeed equivalent to (\ref{eq:Qnm}). Clearly, all these requirements on the measure $\mes$ are met in the present case.

Concerning the admissibility property, observe that the zeros of $q_n$ are contained in $\D$ by the very definition of $g_n$ and their counting measures converge weak$^*$ to the Green equilibrium distribution on $[c,d]$ by \cite[Thm. 2.1]{uBY1}. Thus, the second requirement for admissibility is satisfied. So, it only remains to construct the rearrangements $\Delta_n$ that we shall simply take to be the identity mappings. This way we are required to show that the sums $\sum_{j=1}^n |\phi(\bar\xi_{j,n})-\phi(\xi_{j,n})|$ remain bounded when $n\to\infty$, where $\xi_{j,n}$ are the zeros of $q_n$ and $\phi(\cdot):=\map(1/\cdot)$. Since $\phi$ is holomorphic in $\dr^*$, $\{\xi_{j,n}\}\subset\D$, and $\overline\D\subset\dr^*$, it holds that that
\[
|\phi(z_1)-\phi(z_2)| \leq s_1|z_1-z_2|, \quad z_1,z_2\in\D,
\]
where $s_1 : = \max_\T |\map^\prime| = \max_\T|\phi^\prime| = \max_\D |\phi^\prime|$ by the very definition of $\phi$ and the maximum modulus principle for analytic functions. Hence,
\begin{equation}
\label{eq:sums}
\sum_{j=1}^n |\phi(\bar\xi_{j,n})-\phi(\xi_{j,n})| \leq 2s_1 \sum_{j=1}^n|\im(\xi_{j,n})| \leq 2s_1 \sum_{j=1}^n(\pi-\ang(\xi_{j,n})),
\end{equation}
where $\ang(z):=|\Arg(a-z)-\Arg(b-z)|$, $\Arg(z)\in(-\pi,\pi]$ is the principal branch of the argument of $z$, and we set $\Arg(0)=\pi$. The uniform boundedness of the sums on the right-hand side of (\ref{eq:sums}) was established in \cite[Lem. 3.1]{uBY1}, using in an essential manner that the argument of $\dot\mes$ is of bounded variation, as a prerequisite for the proof of \cite[Thm. 2.1]{uBY1}. This finishes the proof of the lemma.
\end{proof}

\begin{prop}
\label{pr:1}
Under the conditions of either Theorem \ref{thm:mer2} or Theorem \ref{thm:mer3} it holds that
\begin{equation}
\label{eq:errForm}
(\ct-g_n)(z) = \frac{q_{n,m}(z)}{(b_n^2w_nq)(z)}\int\frac{(b_n^2w_nq)(t)}{q_{n,m}(t)}\frac{d\mes(t)}{z-t}, \quad z\in\overline\D\cap\dr.
\end{equation}
\end{prop}
\begin{proof} Let $\{v_n\}$ be a sequence of singular vectors associated to $\{g_n\}$ having  inner-outer factorizations (\ref{eq:sinVecDec}). It was obtained in \cite[Prop. 9.1]{BS02} that
\begin{equation}
\label{eq:hankelonvn}
\begin{array}{ll}
\ho_f(v_n)(\xi) = \sigma_n\overline{\xi}~\overline{\left(b_nj_nw_n^{p^\prime/2}\right)(\xi)} = \sigma_n\left(b_nj_nw_n^{p^\prime/2}\right)^\sigma(\xi), & p>2 \\ 
\ho_f(v_n)(\xi) = \overline{\xi}~\overline{(b_na_n)(\xi)}=\left(b_na_n\right)^\sigma(\xi), & p=2,
\end{array}
\end{equation}
for a.e. $\xi\in\T$, where $j_n$ is some inner function and $a_n\in H^2$. Following the analysis in \cite[Sec. 10]{BS02}, this leads to orthogonality relations of the form
\begin{equation}
\label{eq:auxortho}
\int\frac{(l_{n-1}q_nw_n)(t)}{\widetilde q_n^2(t)}d\mes(t) + \int_\T \frac{(l_{n-1}q_nw_n)(\tau)}{\widetilde q_n^2(\tau)} \frac{p(\tau)}{q(\tau)}\frac{d\tau}{2\pi i} = 0
\end{equation}
for any polynomial $l_{n-1}$, $\deg(l_{n-1})\leq n-1$. In another connection, (\ref{eq:errFunDec}) yields that
\begin{equation}
\label{eq:auxerr}
w_n(\ct-g_n) = \frac{w_n\ho_\ct(v_n)}{v_n} = \frac{\Pj_-(\ct v_n)}{b_n}.
\end{equation}
The right-hand side of (\ref{eq:auxerr}) is holomorphic outside of $\overline\D$ and is vanishing at infinity. So, by the Cauchy theorem it can be written as
\begin{eqnarray}
\frac{\Pj_-(\ct v_n)}{b_n}(z) &=& \frac{1}{b_n(z)}\int_\T\frac{(\ct v_n)(\tau)}{z-\tau}\frac{d\tau}{2\pi i} \nonumber \\
{} &=& \frac{\widetilde q_n(z)}{q_n(z)}\left(\int\frac{(q_nw_n)(t)}{\widetilde q_n(t)}\frac{d\mes(t)}{z-t} + \frac{1}{2\pi i}\int_\T \frac{(q_nw_n)(\tau)}{\widetilde q_n(\tau)} \frac{p(\tau)}{q(\tau)}\frac{d\tau}{z-\tau} \right) \nonumber
\end{eqnarray}
for $|z|>1$. Using (\ref{eq:auxortho}) with $l_{n-1}(t)=(\widetilde q_n(z)-\widetilde q_n(t))/(z-t)$, we get that
\[
\frac{\Pj_-(\ct v_n)}{b_n}(z) = \frac{\widetilde q_n^2(z)}{q_n(z)}\left(\int\frac{(q_nw_n)(t)}{\widetilde q_n^2(t)}\frac{d\mes(t)}{z-t} + \frac{1}{2\pi i}\int_\T \frac{(q_nw_n)(\tau)}{\widetilde q_n^2(\tau)} \frac{p(\tau)}{q(\tau)}\frac{d\tau}{z-\tau} \right).
\]
Applying (\ref{eq:auxortho}) again, now with $l_{n-1}=((qu_{n-m})(z)-(qu_{n-m})(\cdot))/(z-\cdot)$, and using the Cauchy integral formula to get rid of the second integral, we obtain that
\[
\frac{\Pj_-(\ct v_n)}{b_n}(z) = \frac{\widetilde q_n^2(z)}{(u_{n-m}qq_n)(z)} \int\frac{(u_{n-m}qq_nw_n)(t)}{\widetilde q_n^2(t)}\frac{d\mes(t)}{z-t}, \quad |z|>1.
\]
Observe now that the last expression is well-defined as a meromorphic function everywhere in $\dr$. Thus, it follows from (\ref{eq:auxerr}) that (\ref{eq:errForm}) holds.
\end{proof}

\begin{lem}
\label{lem:auxN}
Let $\lambda_n$ be a sequence of Borel complex measures on $[c,d]$ such that
\[
\hat F_n(z) := \int\frac{q(t)}{q_{n,m}(t)}\frac{d\lambda_n(t)}{z-t}, \quad z\in\dr,
\]
converges to some function $F$ locally uniformly in $\dr$. Then
\[
F_n - F = o(1), \quad F_n(z) := \int\frac{d\lambda_n(t)}{z-t}, \quad z\in\dr,
\]
locally uniformly in $\dr$.
\end{lem}
\begin{proof}
Assume first that $q(z)=(z-\eta)^m$. Let $z\in\da$ and $\Gamma_1$ and $\Gamma_2$ be two Jordan curves encompassing $[c,d]$ and $\{\eta\}$, respectively, separating them from each other, and containing $z$ within the unbounded components of their complements. Then
\begin{eqnarray}
\left(\hat F_n\frac{q_{n,m}}{q}\right)(z) &=& \frac{1}{2\pi i}\int_{\Gamma_1}\left(\hat F_n\frac{q_{n,m}}{q}\right)(\tau)\frac{d\tau}{z-\tau} + \frac{1}{2\pi i}\int_{\Gamma_2}\left(\hat F_n\frac{q_{n,m}}{q}\right)(\tau)\frac{d\tau}{z-\tau} \nonumber \\
{} &=& F_n(z) + \frac{1}{(m-1)!}\left(\frac{\hat F_nq_{n,m}}{z-\cdot}\right)^{(m-1)}(\eta), \nonumber
\end{eqnarray}
where we used the Fubini-Tonelli theorem and the Cauchy integral formula. Thus,
\begin{equation}
\label{eq:wierd}
F_n(z) = \left[(q_{n,m}\hat F_n)(z) - T_{\eta,m-1}(q_{n,m}\hat F_n;z)\right]/q(z), \quad z\in\dr,
\end{equation}
where $T_{\eta,m-1}(f;\cdot)$ is the $(m-1)-$st partial sum of the Taylor expansion of $f$ at $\eta$. By (\ref{eq:Qnm}) the polynomials $T_{\eta,m-1}(q_{n,m}\hat F_n;\cdot)$ converge to zero as $n$ tends to infinity and the claim of the lemma follows by the maximum modulus principle for analytic functions. By partial fraction decomposition, the case of a general $q$ is no different.
\end{proof}
\begin{proof}[Proof of Theorem \ref{thm:mer1}] Let
\begin{equation}
\label{eq:dnun}
d\nu_n := \frac{q_{n+m,m}qw_{n+m}}{\widetilde q_{n+m}^2}d\mes = \frac{q_{n+m,m}qw_{n+m}h\hbar}{\widetilde q_{n+m}^2}d\ed.
\end{equation}
Then we get from (\ref{eq:auxortho}) applied with $n$ replaced by $n+m$ and $l_{n+m-1}(t)=t^jq(t)$, $j=0,\ldots,n-1$, that
\begin{equation}
\label{eq:orthoun}
\int t^ju_n(t)d\nu_n(t) = 0.
\end{equation}
So, the asymptotic behavior of $u_n$ is governed by Theorem \ref{thm:ortho}, applied with $v_n=\widetilde q_n^2$ and $h_n=q_{n+m,m}qw_{n+m}$, due to Lemma \ref{lem:auxS} and the fact that $\{w_n\}$ is a normal family in $\D$ none of which limit points has zeros. The latter was obtained in \cite[Lem. 3.4]{uBY1} under the mere assumption that $\mu$ has infinitely many points in the support and an argument of bounded variation.

In another connection, observe that
\[
b_n\left(b_nj_nw_n^{p^\prime/2}\right)^\sigma = |b_n|^2\left(j_nw_n^{p^\prime/2}\right)^\sigma = \left(j_nw_n^{p^\prime/2}\right)^\sigma \quad \mbox{on} \quad \T
\]
and that $\left(j_nw_n^{p^\prime/2}\right)^\sigma$ is the trace of a function from $\bar H_0^2$. Thus, it follows from (\ref{eq:hankelonvn}) that
\[
\Pj_-(b_n\ho_\ct(v_n)) = \sigma_n\Pj_-\left(b_n\left(b_nj_nw_n^{p^\prime/2}\right)^\sigma\right) = \sigma_n\left(j_nw_n^{p^\prime/2}\right)^\sigma.
\]
It is also readily checked that
\[
\Pj_-(b_n\ho_\ct(v_n))(z) = \Pj_-(b_n\Pj_-(\ct v_n))(z) = \Pj_-(\ct b_nv_n)(z) = \int_\T\frac{(\ct b_nv_n)(\tau)}{z-\tau}\frac{d\tau}{2\pi i}
\]
for $|z|>1$. Hence, we derive by using the Fubini-Tonelli theorem that
\begin{equation}
\label{eq:ExtendOuter}
\frac{\sigma_n}{\gamma_{n-m}}\left(j_nw_n^{p^\prime/2}\right)^\sigma(z) = \frac{1}{\gamma_{n-m}} \left( \int \frac{(b_n^2w_n)(t)}{z-t}d\mes(t) + \int_\T\frac{(b_n^2w_n)(\tau)}{z-\tau}\frac{p(\tau)}{q(\tau)}\frac{d\tau}{2\pi i} \right),
\end{equation}
for $|z|>1$, where $\gamma_n$ has the same meaning as in Theorem \ref{thm:ortho}. As the right-hand side of (\ref{eq:ExtendOuter}) is defined everywhere in $\da$, the restriction $|z|>1$ is no longer necessary. This, in particular, implies that $j_n$ is a finite Blaschke product as neither singular inner factors nor infinite Blaschke products can be extended even continuously on $\T$. However, notice that first we should evaluate the second integral on the right-hand side of (\ref{eq:ExtendOuter}) by the residue formula and only then remove the restriction $|z|>1$. Clearly, this integral represents a rational function vanishing at infinity whose poles are those of $q$. It is also easy to observe that if
\begin{equation}
\label{eq:aa1}
\gamma_{n-m}^{-1}(b_n^2w_np)^{(k)}(\eta) = o(1), \quad k=0,\ldots,m(\eta)-1, \quad \mbox{and} \quad (b_n^2w_np)(\eta) \neq 0,
\end{equation}
this rational function converges to zero locally uniformly in $\overline\C\setminus\z(q)$ and has poles of exact multiplicity $m(\eta)$ at each $\eta\in\z(q)$. Now, we have by (\ref{eq:errForm}), \eqref{eq:secondkind}, and (\ref{eq:nhop}) that
\begin{equation}
\label{eq:aa2}
(\ct-g_n) = \frac{q_{n,m}}{q}\frac{\gamma_{n-m}y_n}{b_n^2w_n}, \quad y_n := \frac{1}{\gamma_{n-m}}u_{n-m}R_{n-m},
\end{equation}
and
\begin{equation}
\label{eq:asympyn}
y_n(z) := \int\frac{u_{n-m}^2(t)}{\gamma_{n-m}}\frac{d\nu_{n-m}(t)}{z-t} = \int\frac{(b_n^2w_n)(t)}{\gamma_{n-m}} \frac{q(t)}{q_{n,m}(t)}\frac{d\mes(t)}{z-t} = \frac{1+o(1)}{\sr(z)}
\end{equation}
locally uniformly in $\dr$. Then we get from (\ref{eq:aa2}) that
\[
(q_{n,m}y_n)(z) = \gamma_{n-m}^{-1}(b_n^2w_n)(z)(\ct_\mu q+p)(z) + \gamma_{n-m}^{-1}(qb_nw_nh_n)(z), \quad z\in\dr,
\]
where $\ct=\ct_\mu+p/q$ and $g_n=h_n/b_n$, and therefore for $\eta\in\z(q)$ we obtain
\begin{equation}
\label{eq:aa7}
(q_{n,m}y_n)^{(k)}(\eta) = \gamma_{n-m}^{-1}(b_n^2w_np)^{(k)}(\eta), \quad k=0,\ldots,m(\eta)-1.
\end{equation}
Hence, the first part of (\ref{eq:aa1}) follows from (\ref{eq:Qnm}) and the normality of $\{y_n\}$, which is immediately deduced from \eqref{eq:asympyn}. The second part of \eqref{eq:aa1} holds since \eqref{eq:asympyn} and \eqref{eq:aa7}, applied with $k=0$, yield that for all $n$ large enough we have
\[
0 \neq y_n(\eta) = \gamma_{n-m}^{-1}(b_nw_npu_{n-m}/\widetilde q_n))(\eta)
\]
and therefore $(b_nw_np)(\eta)$ as well as $(b_n^2w_np)(\eta)$ cannot vanish.

In another connection, (\ref{eq:asympyn}) and Lemma \ref{lem:auxN} yield that
\begin{equation}
\label{eq:aa3}
\frac{1}{\gamma_{n-m}}\int \frac{(b_n^2w_n)(t)}{z-t}d\mes(t) = \frac{1+o(1)}{\sr(z)}
\end{equation}
locally uniformly in $\dr$. Thus, combining (\ref{eq:aa1}) and (\ref{eq:aa3}) with (\ref{eq:ExtendOuter}), we get that
\[
\frac{\sigma_n}{\gamma_{n-m}}\left(j_nw_n^{p^\prime/2}\right)^\sigma = \frac{1+o(1)}{\sr} + \frac{\ell_n}{q}
\]
locally uniformly in $\dr$, where $\deg(\ell_n)<m$, the polynomials $\ell_n$ are coprime with $q$, and converge to zero locally uniformly in $\C$ when $n\to\infty$. Equivalently, we have that
\begin{equation}
\label{eq:aa4}
\frac{\sigma_n}{\bar\gamma_{n-m}}j_nw_n^{p^\prime/2}\widetilde\sr = 1+o(1) + \frac{\widetilde\ell_n\widetilde\sr}{\widetilde q} = 1 +o(1),
\end{equation}
where the first $o(1)$ holds locally uniformly in $\dr^*$, the second one holds locally uniformly in $\das$, and $\widetilde\ell_n(z) := z^{m-1}\overline{\ell_n(1/\bar z)}$, $\deg(\widetilde\ell_n)<m$ since $\deg(\ell_n)<m$. This, in particular, implies that the Blaschke products $j_n$ are identically 1 for all $n$ large enough since the right-hand side of (\ref{eq:aa4}) cannot vanish in $\overline\D$ for such $n$. Finally, recall that by its very definition $w_n^{p^\prime/2}$ has unit $L^2$ norm. Therefore, deformation of the integral on $\T$ to $[c,d]$ covered twice by the Cauchy integral formula yields that
\begin{eqnarray}
\left(\frac{\sigma_n}{|\gamma_{n-m}|}\right)^2 &=& \frac{1}{2\pi} \int_\T \frac{|d\tau|}{|\widetilde\sr(\tau)|^2} + o(1) = \frac{1}{2\pi i}\int_\T\frac{d\tau}{(\sr\widetilde\sr)(\tau)} + o(1) \nonumber \\
\label{eq:aa5}
{} &=& \frac{1}{\pi i}\int_{[c,d]}\frac{dt}{(\sr^-\widetilde\sr)(t)} + o(1) = \mathcal{T}^{-2} + o(1)
\end{eqnarray}
since $\widetilde\sr(\tau)=\tau\overline{\sr(\tau)}$ on $\T$, $\sr^-= -\sr^+= -i|\sr^\pm|$ on $[c,d]$, and on account of (\ref{eq:ged}). Thus, (\ref{eq:mer1}) shall follow from \eqref{eq:aa4} and \eqref{eq:aa5} with $l_n:=\mathcal{T}\widetilde\ell_n$ upon showing that
\begin{equation}
\label{eq:aa6}
\gamma_n/|\gamma_n| = 1 + o(1).
\end{equation}
The latter is an easy consequence of (\ref{eq:aa4}) since $\widetilde\sr(0)=1$ and $w_n(0)>0$.
\end{proof}

In the next proposition of technical nature, we define a special sequence of Szeg\H{o} functions for the condenser $\dr\cap\dr^*$ that appears in Theorem \ref{thm:mer2}.

\begin{prop}
\label{pr:2}
For each $p\in[2,\infty]$ there exists a normal family of non-vanishing functions in $\dr\cap\dr^*$, denoted by $\{\szc_n\}$, such that
\[
\hat\gmc|\szc_n^\pm| = \left|(r_n/r)^\pm\szf_{b^2w\dot\mes}^\pm\right| \quad \mbox{on} \quad [c,d], \quad \gm_{\widetilde q_n/\widetilde q}~\frac{(\szf_{b^2w\dot\mes}r_n)(1)}{(r\szc_n)(1)} > 0,
\]
where $r_n:=r_n(\widetilde q_n;\cdot)$, $r:=r_m(\widetilde q;\cdot)$, $\hat\gmc^2=\gmc/|\gm_{b^2w\dot\mes}|$, and $\gmc$ is given by \eqref{eq:m}. Moreover, each $\szc_n$ satisfies $\szc_n(z)\overline{\szc_n(1/\bar z)} = 1$, $z\in\dr\cap\dr^*$, has continuous traces on each side of $[c,d]$, and has winding number zero on any curve separating $[c,d]$ from $[c,d]^*$.
\end{prop}
\begin{proof}
The concept of Szeg\H{o} function for a condenser initially was developed in \cite[Thm. 1.6]{LS_PAT92} in the case of an annulus. It was shown that if $Y$ is a continuous (strictly) positive function on $\T_s$, $s<1$, then there exists a function $u$, harmonic in $\A_s$, such that $u=\frac12\log(Y/\gmc_Y)$ on $\T_s$, $u\equiv0$ on $\T$, and $u(1/\bar z)=-u(z)$, $z\in\A_s$, where $\gmc_Y:=\exp\left\{\int_{\T_s} \log Y(\tau)\frac{|d\tau|}{2\pi s}\right\}$ is the geometric mean of $Y$. Moreover, it was shown that $u$ has single-valued harmonic conjugate $v$. Moreover, the latter is unique up to an additive constant. Finally, it was deduced that $\szc_s(Y;\cdot):=\exp\{u+iv\}$, the Szeg\H{o} function of $Y$ for $\A_s$, is a non-vanishing holomorphic function in $\A_s$ such that $\gmc_Y|\szc_s(Y;\cdot)|^2=Y$ on $\T_s$, $\szc_s(Y;z)\overline{\szc_s(Y;1/\bar z)}=1$, $z\in\A_s$, and $\szc_s(Y;\cdot)$ is an outer function in $\A_s$ with zero winding number on any curve in $\A_s$. The latter was not explicitly stated in \cite{LS_PAT92} but clearly holds since $\log|\szc_s(Y;\cdot)|=u$ and therefore it is the integral of its boundary values against the harmonic measure on $\partial\A_s$ while $\Arg(\szc_s(Y;\cdot))=v$, which has zero increment on any curve in $\A_s$. Obviously, the Szeg\H{o} function for $\A_s$ is unique up to a multiplicative unimodular constant.

Let now $y^+$ and $y^-$ be two continuous positive functions on $[c,d]$ whose values at the endpoints coincide. We define the geometric mean and the Szeg\H{o} function of the pair $\mathfrak{y}:=(y^+,y^-)$ for the condenser $\dr\cap\dr^*$ by
\[
\gmc_\mathfrak{y} := \gmc_Y \quad \mbox{and} \quad \szc_\mathfrak{y}(z) := \szc_\rho(Y;\mpc(z)), \quad z\in\dr\cap\dr^*,
\]
respectively, where $Y(\tau) = y^\pm(\mpc^{-1}(\tau))$, $\tau\in\T^\pm_\rho$. It is an immediate consequence of the corresponding properties of $\szc_\rho(Y;\cdot)$ and $\mpc$ that $\szc_\mathfrak{y}$ is outer, has non-tangential continuous boundary values on both sides of $[c,d]$ and $[c,d]^*$ whenever $\mathfrak{y}$ is a Dini-continuous pair\footnote{This means that $Y$ and therefore $\log Y$ are Dini-continuous as $\mpc^{-1}$ is Lipschitz on $\T_\rho$. Hence, the boundary values of $v$ are continuous on $\partial\A_\rho$ \cite[Ch. III]{Garnett}.}, has winding number zero on any curve in $\dr\cap\dr^*$, and satisfies
\[
|\szc_\mathfrak{y}^\pm(t)|^2 = \left\{
\begin{array}{ll}
y^\pm(t)/\gmc_\mathfrak{y},   & t\in [c,d], \smallskip \\
\gmc_\mathfrak{y}/y^\pm(1/t), & t\in [c,d]^*,
\end{array}
\right.
\]
and $\szc_\mathfrak{y}(z)\overline{\szc_\mathfrak{y}(1/\bar z)} = 1$, $z\in\dr\cap\dr^*$. 

Now, put $y^\pm := |b^2w\dot\mes|$ and $y_n^\pm := \left|(r_n/r)^\pm \szf_{b^2w\dot\mes}^\pm\right|^2$. Observe that in this case
\[
y_n^+(t)y_n^-(t) = |\szf_{b^2w\dot\mes}^+(t)\szf_{b^2w\dot\mes}^-(t)|^2 = \left|\frac{(b^2w\dot\mes)(t)}{\gm_{b^2w\dot\mes}}\right|^2 = \frac{y^+(t)y^-(t)}{|\gm_{b^2w\dot\mes}|^2},
\]
$t\in[c,d]$, by \eqref{eq:szegodecomp} and \eqref{eq:rkpm}. Then for $\mathfrak{y}_n:=(y_n^+,y_n^-)$ we get that
\begin{eqnarray}
\log\gmc_{\mathfrak{y}_n} &=& \int_{\T^+_\rho} \log[Y_n(\tau)Y_n(\bar\tau)]\frac{|d\tau|}{2\pi\rho} = \int_{\T^+_\rho} \log[y_n^+(t)y_n^-(t)]\frac{|d\tau|}{2\pi\rho} \nonumber \\
{} &=& \int_{\T^+_\rho} \log\left[\frac{y^+(t)y^-(t)}{|\gm_{b^2w\dot\mes}|^2}\right]\frac{|d\tau|}{2\pi\rho} = \log\gmc_\mathfrak{y} - \log|\gm_{b^2w\dot\mes}|, \nonumber
\end{eqnarray}
where $t=\mpc^{-1}(\tau)=\mpc^{-1}(\bar\tau)$. Further, by \eqref{eq:ged} and \eqref{eq:m}, we have that
\[
\log\gmc_\mathfrak{y} = \int_{[c,d]} \log\left|(b^2w\dot\mes)(t)\right| \frac{d|\mpc^+(t)|}{\pi\rho} = \log\gmc.
\]
So, $\gmc_{\mathfrak{y}_n} = \gmc/|\gm_{b^2w\dot\mes}| = \hat\gmc^2$ and therefore the claim of the proposition follows by setting $\szc_n := \szc_{\mathfrak{y}_n}$ with the chosen normalization (recall that the functions $\szc_{\mathfrak{y}_n}$ are uniquely defined up to a unimodular constant).
\end{proof}

The following Lemma was proved in \cite[Lem. 4.7]{BStW01}.

\begin{lem}
\label{lem:BStW}
Let $U$ be a domain, $\partial U=K_1\cup K_2$, $K_1$ and $K_2$ be two disjoint compact sets in $\C$, and $u$ be a harmonic function in $U$. Assume that
\[
\int_\Gamma \frac{\partial u}{\partial n}ds = 0,
\]
where $\partial/\partial n$ and $ds$ are, respectively, the normal derivative and the arclength differential on $\Gamma$, and the latter is an oriented smooth Jordan curve that separates $K_1$ from $K_2$, has winding number 1 with respect to any point of $K_1$, and winding number 0 with respect to any point of $K_2$. Then
\[
\sup_{z^\prime\in K_1} \limsup_{z\to z^\prime,~ z\in U} u(z) \geq \inf_{z^\prime\in K_2} \liminf_{z\to z^\prime,~z\in U} u(z)
\]
and the same relation holds with $K_1$ and $K_2$ interchanged.
\end{lem}

\begin{proof}[Proof of Theorem \ref{thm:mer2}] It was shown in Lemma \ref{lem:auxS} that $q_n$ can be written as $u_{n-m}q_{n,m}$ and the behavior of $u_{n-m}$ is governed by Theorem \ref{thm:ortho} with $\nu_n$ defined in (\ref{eq:dnun}). Thus, we have from (\ref{eq:nhop}) that
\begin{equation}
\label{eq:zz1}
1+o(1) = u_{n-m}\szf_{n-m} = u_{n-m}(\cp\map)^{n-m}\szf_{q\dot\mes/\widetilde q_n^2}\szf_{q_{n,m}w_n}
\end{equation}
locally uniformly in $\dr$. Further, since $q_{n,m}w_n$ converges to $qw$ uniformly on $[c,d]$ by Lemma \ref{lem:auxS} and Theorem \ref{thm:mer1} we get that $\szf_{q_{n,m}w_n} = [1+o(1)]\szf_{qw}$ uniformly in $\dr$ and therefore we obtain from (\ref{eq:zz1}) that
\begin{equation}
\label{eq:zz2}
1 + o(1) = u_{n-m}\cp^{n-m}\szf_{b^2w\dot\mes}~\map^{n-m}\szf_{\widetilde q^2/\widetilde q_n^2}
\end{equation}
locally uniformly in $\dr$, where $b^2=q^2/\widetilde q^2$. Now, it follows from (\ref{eq:szegopoly}) that
\[
\szf_{\widetilde q^2/\widetilde q_n^2} = \szf_{\widetilde q/\widetilde q_n}^2 = \frac{\map^{m-n}}{\gm_{\widetilde q/\widetilde q_n}} \frac{\widetilde q}{\widetilde q_n} \frac{r_n}{r},
\]
where $r:=r_m(\widetilde q;\cdot)$ and $r_n:=r_n(\widetilde q_n;\cdot)$ as in Proposition \ref{pr:2}. Hence, we deduce from (\ref{eq:zz2}) that
\begin{equation}
\label{eq:zz3}
1 + o(1) = \frac{b_n}{b}\frac{q}{q_{n,m}} \frac{\cp^{n-m}}{\gm_{\widetilde q/\widetilde q_n}} \szf_{b^2w\dot\mes}\frac{r_n}{r}
\end{equation}
locally uniformly in $\dr$. Then Lemma \ref{lem:auxS} implies that
\begin{equation}
\label{eq:zz4}
\lambda_nX_n\szc_nb_n\mpc^{m-n}b^{-1} = 1+o(1),
\end{equation}
locally uniformly in $\da$, where
\[
\lambda_n := \frac{\hat\gmc(\cp\rho)^{n-m}}{\gm_{\widetilde q/\widetilde q_n}} \quad \mbox{and} \quad X_n(z) := \frac{\szf_{b^2w\dot\mes}}{\hat\gmc\szc_n} \frac{r_n}{r} \left(\frac{\mpc}{\rho}\right)^{n-m}.
\]
Now, we shall show that
\begin{equation}
\label{eq:zz5}
\lambda_nX_n = 1+o(1) \quad \mbox{uniformly in} \quad \overline\D.
\end{equation}
Observe, that
\begin{equation}
\label{eq:zz6}
|X_n^\pm| = \left|\left(\frac{\szf_{b^2w\dot\mes}}{\hat\gmc\szc_n} \frac{r_n}{r}\right)^\pm\right| \equiv 1 \quad \mbox{on} \quad [c,d]
\end{equation}
by the very definition of $\szc_n$ (see Proposition \ref{pr:2}). Moreover, since the zeros of $r$, $\z(\widetilde q)$, lie outside of $\overline\D$ and the zeros of $r_n$, $\z(\widetilde q_n)$, approach $[c,d]^*$ and $\z(\widetilde q)$ by Lemma \ref{lem:auxS} and Theorem \ref{thm:ortho}, the functions $X_n$ are zero free in some neighborhood of $\overline\D$, where the values on $[c,d]$ are twofold.  Further, the winding number of $X_n$ along any smooth Jordan curve encompassing $[c,d]$ in $\D$ is equal to zero. Indeed, the winding number of $\szf_{b^2w\dot\mes}/\szc_n$ on such a curve is zero by the properties of Szeg\H{o} functions, $r_n/r$ has winding number $m-n$ since it is meromorphic outside of $[c,d]$ with $n$ zeros and $m$ poles outside of $\overline\D$, and it follows from \cite[Ch. VI]{Nehari} that $\mpc$ has winding number one on any such curve. Thus, $\log X_n$ are well-defined holomorphic functions in $\overline\D\setminus[c,d]$. In turn, this means that $\log|X_n|$ satisfies the conditions of Lemma \ref{lem:BStW} with $U=\D\setminus[c,d]$. Applying this lemma in both directions, we get from (\ref{eq:zz6}) that
\begin{equation}
\label{eq:zz7}
\inf_\T|X_n|\leq \sup_{[c,d]}|X_n^\pm| = 1 = \inf_{[c,d]} |X_n^\pm| \leq \sup_\T|X_n|.
\end{equation}
In another connection, (\ref{eq:zz3}) and \eqref{eq:Qnm} yield that uniformly on $\T$ we have
\begin{equation}
\label{eq:zz8}
|X_n| = \left|\frac{\gm_{\widetilde q/\widetilde q_n}}{\hat\gmc(\cp\rho)^{n-m}} ~ \frac{\cp^{n-m}}{\gm_{\widetilde q/\widetilde q_n}} \szf_{b^2w\dot\mes}\frac{r_n}{r}\right| = \frac{1+o(1)}{|\lambda_n|}
\end{equation}
since $\szc_n$, $\mpc$, $b_n$, and $b$ are unimodular on $\T$. Combining (\ref{eq:zz7}) with (\ref{eq:zz8}), we get that $|\lambda_n|=1+o(1)$ and therefore
\begin{equation}
\label{eq:5stars}
|\lambda_nX_n| = \left|\frac{(\cp\mpc)^{n-m}}{\gm_{\widetilde q/\widetilde q_n}} \frac{\szf_{b^2w\dot\mes} r_n}{\szc_n r}\right| = 1 + o(1) \quad \mbox{uniformly in} \quad \overline\D
\end{equation}
by the maximum principle for harmonic functions applied to $\pm\log|X_n|$ in $\D\setminus[c,d]$. Hence, $\{\lambda_nX_n\}$ is a normal family of harmonic functions in $\D\setminus[c,d]$ and all the limit points of this family are the unimodular constants. Therefore (\ref{eq:zz5}) follows from the normalization of $\szc_n$ (see Proposition \ref{pr:2}) and the fact that $\mpc(1)$=1.

Clearly, we can rewrite (\ref{eq:zz4}) with the help of (\ref{eq:zz5}) as
\[
\szc_nb_n\mpc^{m-n}b^{-1} = 1+o(1) \quad \mbox{uniformly on compact subsets of} \quad \overline\D\cap\da.
\]
Now, recall that $\overline{\szc_n(1/\bar\cdot)} = 1/\szc_n$. Moreover, the same property holds for $b_n$, $b$, and $\mpc$. Thus,
\[
\szc_nb_n\mpc^{m-n}b^{-1} =  1/\overline{(\szc_nb_n\mpc^{m-n}b^{-1})(1/\bar\cdot)} = 1+o(1)
\]
uniformly on closed subsets of $\das\setminus\D$ and (\ref{eq:mer2}) follows.

It only remains to prove (\ref{eq:mer3}) and (\ref{eq:mer4}). By the very definition of $\gamma_n$ in Theorem \ref{thm:ortho}, we have that
\[
|\gamma_{n-m}| = \frac{2|\gm_{\dot\nu_{n-m}}|}{\cp^{2(n-m)}} = [1+o(1)] \frac{2|\gm_{b^2w\dot\mes}||\gm_{\widetilde q/\widetilde q_n}^2|}{\cp^{2(n-m)}}
\]
by Lemma \ref{lem:auxS} and limit (\ref{eq:mer1}). Further, the very definitions of $\hat\gmc$ and $\lambda_n$ yield that
\[
\frac{2|\gm_{b^2w\dot\mes}||\gm_{\widetilde q/\widetilde q_n}^2|}{\cp^{2(n-m)}} =  \frac{2|\gm_{b^2w\dot\mes}|\hat\gmc^2}{\rho^{2(m-n)}} ~ \frac{|\gm_{\widetilde q/\widetilde q_n}|^2}{\hat\gmc^2(\cp\rho)^{2(n-m)}} = \frac{2\gmc\rho^{2(n-m)}}{|\lambda_n|^2}.
\]
Since $|\lambda_n|=1+o(1)$, it holds that
\begin{equation}
\label{eq:zz9}
\gamma_{n-m} = [1+o(1)]|\gamma_{n-m}| = [1+o(1)]2\gmc\rho^{2(n-m)},
\end{equation}
where we used (\ref{eq:aa6}). Thus, (\ref{eq:mer3}) follows from (\ref{eq:aa5}). Finally, we deduce from (\ref{eq:aa2}) and (\ref{eq:asympyn}) that
\[
(\ct-g_n) = [1+o(1)] \frac{q_{n,m}}{q} \frac{\gamma_{n-m}}{b_n^2w_n}\frac{1+o(1)}{\sr}
\]
uniformly on compact subsets of $\da\cap\overline\D$. Since $q_{n,m}/q\to1$ by Lemma \ref{lem:auxS}, $w_n\to w$ by (\ref{eq:mer1}), and using (\ref{eq:zz9}) with (\ref{eq:mer2}), (\ref{eq:mer4}) follows.
\end{proof}

\begin{proof}[Proof of Theorem \ref{thm:mer3}] Recall that Lemma \ref{lem:auxS} holds under the conditions of this theorem and the Szeg\H{o} functions $\szc_n$ exist for $p=2$ as well. As the starting point of the proof of Theorem \ref{thm:mer2} was the application of (\ref{eq:nhop}) with $\nu_n$ defined in (\ref{eq:dnun}), all we need to do is to show that the conditions of Theorem \ref{thm:ortho} still hold under the present assumptions. This is tantamount to show that $\Upsilon_x$, defined in (\ref{eq:Upsilon}), is minorized by $\Psi_x$, defined in the statement of the theorem, for all $x\in\x$. In other words, that
\[
\Psi_x \leq \liminf_{n\to\infty} \min_{\pm}\left\{\left|(r_n\scf_{qq_{n,m}h})^\pm(x)\right|\right\}, \quad x\in\x.
\]
Equivalently, we need to show that
\[
- \liminf_{n\to\infty} \log|r_n^\pm(x)|^2 = \limsup_{n\to\infty} \log|r_n^\mp(x)|^2 \leq \frac{4s_1(V_h+2m\pi)}{1-s_0}, \quad x\in\x,
\]
by Lemma \ref{lem:auxS} and the definition of $\Psi_x$. Let, as in the proof of Lemma \ref{lem:auxS}, $\xi_{j,n}$, $j=1,\ldots,n$, be the zeros of $q_n$. Then we get from (\ref{eq:sums}) that
\[
\sum_{j=1}^n|\map(1/\bar\xi_{j,n})-\map(1/\xi_{j,n})| \leq 2s_1 \sum_{j=1}^n(\pi-\ang(\xi_{j,n})) \leq 2s_1(V_h+2m\pi),
\]
where we used \cite[Lem. 3.2]{uBY1} for the last inequality. Put $\hat r_n:=r_n\circ\map^{-1}$. Then, exactly as we did to prove (\ref{eq:boundrnhat}), we obtain that
\[
\log|\hat r_n| \leq 2s_1(V_h+2m\pi)/(1-s_0),
\]
which finishes the proof of the theorem.
\end{proof}

To prove Theorem \ref{thm:mer4} we need the following lemma.

\begin{lem}
\label{lem:auxR}
Let $R=P/Q$ be a rational function of degree $d$, $\zeta\in\C$, and $\delta>0$. Assume further that $P$ and $Q$ have no zeros in $\{z:|z-\zeta|\leq\delta\}$. Then for any $k<d$, $k\in\N$,  there exists $c_k=c_k(\delta)$ independent of $R$ such that
\[
\left|R^{(k)}(\zeta)/R(\zeta)\right| \leq c_kd^k.
\]
\end{lem}
\begin{proof}
Clearly, if $T$ is a polynomial of degree at most $d$ with no zeros in $\{z:|z-\zeta|\leq\delta\}$, then
\[
\left|\frac{T^{(j)}(\zeta)}{T(\zeta)}\right| \leq \frac{d\cdot\ldots\cdot(d-j+1)}{\delta^j} \leq \left(\frac{d}{\delta}\right)^j, \quad j=1,\ldots,k.
\]
Thus, it can be checked that
\[
\left|T(\zeta)\left(\frac{1}{T(\zeta)}\right)^{(j)}\right| = \left|\sum_{l=1}^j\sum_{\sum d_i=l}\prod_{\sum s_id_i=j} c_{l,\{d_i\},\{s_i\}}\left(\frac{T^{(s_i)}(\zeta)}{T(\zeta)}\right)^{d_i}\right| \leq c_j^* d^j,
\]
$j=1,\ldots,k$, where coefficients $c_{l,\{d_i\},\{s_i\}}$ do not depend on $T$. Then
\[
\left|\frac{R^{(k)}(\zeta)}{R(\zeta)}\right| = \left|\sum_{j=1}^k \binom kj \frac{P^{(j)}(\zeta)}{P(\zeta)} Q(\zeta) \left(\frac{1}{Q(\zeta)}\right)^{(k-j)}\right| \leq \sum_{j=1}^k \binom kj \frac{c_{k-j}^*d^k}{\delta^j} =: c_k d^k. \nonumber
\]
\end{proof}

\begin{proof}[Proof of Theorem \ref{thm:mer4}] As the forthcoming analysis is local around $\eta$, we may suppose without loss of generality that $m(\eta)=m$, i.e. $\eta$ is the only zero of $q$.

Exactly as in (\ref{eq:aa7}), we obtain that
\begin{equation}
\label{eq:pp3}
y_n^{(k)}(\eta) = (Y_nq_{n,m})^{(k)}(\eta), \quad k=0,\ldots,m-1, \quad Y_n:= \frac{u_{n-m}^2w_np}{\gamma_{n-m}\widetilde q_n^2}.
\end{equation}
It is apparent from \eqref{eq:zz4}, \eqref{eq:zz5}, and foremost (\ref{eq:zz3}), which holds locally uniformly in $\dr$ rather then $\da$, that
\[
\frac{u_{n-m}}{\widetilde u_{n-m}}\frac{\szc_n}{\mpc^{n-m}} = 1 + o(1) \quad \mbox{locally uniformly in} \quad \D\cap\dr.
\]
So, we see using \eqref{eq:mer1}, \eqref{eq:Qnm}, \eqref{eq:aa5} with \eqref{eq:aa6}, and \eqref{eq:mer3} that
\begin{equation}
\label{eq:pp5}
Y_n = [1+o(1)]~ \frac{\mpc^{2(n-m)}}{\mathcal{T}\sigma_n}\frac{wp}{\szc_n^2\widetilde q^2} = \left[\frac{1}{2\gmc}+o(1)\right] \left(\frac{\mpc}{\rho}\right)^{2(n-m)}\frac{wp}{\szc_n^2\widetilde q^2}
\end{equation}
uniformly in some neighborhood of $\eta$. Thus, we obtain from (\ref{eq:pp3}) with $k=0$ that
\begin{equation}
\label{eq:pp6}
\chi_n^mq_{n,m}(\eta) = -1, \;\;\; \chi_n^m := -Y_n(\eta)/y_n(\eta),
\end{equation}
where for each $n$ we fixed an arbitrary root $\chi_n$. Observe also that $\chi_n$ tends to infinity geometrically fast by (\ref{eq:pp5}) since $|\mpc(\eta)|>\rho$ and the boundedness of $\{|y_n(\eta)|\}$, which is apparent from \eqref{eq:asympyn}. By putting $k=1$ in (\ref{eq:pp3}), we see that
\[
\chi_n^{m-1}q_{n,m}^\prime(\eta) = \frac{1}{\chi_n}\left(\frac{Y_n^\prime(\eta)}{Y_n(\eta)} - \frac{y_n^\prime(\eta)}{y_n(\eta)}\right) = o(1)
\]
since $\{y_n^\prime(\eta)/y_n(\eta)\}$ is a convergent sequence by \eqref{eq:asympyn} and $Y_n$ are rational functions, which do not vanish in some fixed neighborhood of $\eta$, multiplied by $w_n$, which form a convergent sequence by \eqref{eq:mer1},  the numbers $|(Y^\prime_n/Y_n)(\eta)|$ grow linearly with $n$ by Lemma \ref{lem:auxR} while $1/\chi_n$ decays exponentially. Continuing by induction, we get
\begin{equation}
\label{eq:pp4}
\chi_n^{m-k}q_{n,m}^{(k)}(\eta) = \sum_{j=1}^k  \binom kj \frac{Y_n^{(j)}(\eta)}{Y_n(\eta)} \frac{\chi_n^{m-k+j}q_{n,m}^{(k-j)}(\eta)}{\chi_n^j} - \frac{1}{\chi_n^k} \frac{y_n^{(k)}(\eta)}{y_n(\eta)} = o(1),
\end{equation}
for any $k=2,\ldots,m-1$. Hence, we deduce from (\ref{eq:pp6}) and (\ref{eq:pp4}) that
\[
\prod_{k=1}^m\left(z+\chi_n(\eta-\eta_{k,n})\right) = z^m + \sum_{k=0}^{m-1} \chi_n^{m-k} q_{n,m}^{(k)}(\eta) z^k = z^m + o(1) - 1,
\]
uniformly in some neighborhood of $\eta$. In particular, this means that
\[
\eta_{k,n} = \eta + \frac{1+\delta_{k,n}}{\chi_n}\exp\left\{\frac{2\pi ki}{m}\right\}, \quad k=1,\ldots,m,
\]
where $\delta_{k,n}=o(1)$ for each $k$ and $\prod_{k=1}^m(1+\delta_{k,n})=1$. By setting
\[
\left(A_{k,n}^\eta\right)^m := \frac{1+\delta_{k,n}}{\chi_n^m} \left(\frac{\rho}{\mpc(\eta)}\right)^{2(m-n)} = -[1+\delta_{k,n}] \left(\frac{\rho}{\mpc(\eta)}\right)^{2(m-n)} \frac{y_n(\eta)}{Y_n(\eta)},
\]
we see that (\ref{eq:mer5}) follows. The boundedness of $\{\max_k|A_{k,n}^\eta|\}$ is a consequence of (\ref{eq:pp5}) and \eqref{eq:asympyn}.
\end{proof}

\section{Proofs of Theorems \ref{thm:pade1} and \ref{thm:pade2}}
\label{sec:proofs1}

\begin{proof}[Proof of Theorem \ref{thm:pade1}] Let $q_n$ be the denominators of $\Pi_n$. We start by showing that
\begin{equation}
\label{eq:qnm}
q_n = u_{n-m}q_{n,m}, \quad q_{n,m}=(1+o(1))q,
\end{equation}
locally uniformly in $\C\setminus\z(q)$. This follows from \cite[Thm. 2.4]{BY09} in the same fashion as (\ref{eq:Qnm}) followed from \cite[Thm. 2.4]{uBY1}. The requirements placed on $\mes$ are the same, so they are satisfied. However, in \cite[Thm. 2.4]{BY09} there are also restrictions placed on the interpolation schemes. Namely, an interpolation scheme $\E$ should be such that $\supp(\E)\cap([c,d]\cup\z(q))=\emptyset$, the probability counting measures of points in $E_n$ would converge to some Borel measure with finite logarithmic energy, and the argument functions of polynomials $v_n$ would have uniformly bounded derivatives on $[c,d]$.

Clearly, the first two requirement placed on the interpolation scheme is the second requirement of the admissibility property. Hence, we only need to show the uniform boundedness of the derivatives of the arguments of $v_n$. Clearly, it amounts to show that
\begin{equation}
\label{eq:limarg}
\limsup_{n\to\infty}\frac12\left\|\im\left(\frac{v_n^\prime}{v_n}\right)\right\|_{[c,d]} = \limsup_{n\to\infty}\frac12\left\|\sum_{e\in E_n\cap\C}\im\left(\frac{1}{\cdot-e}\right)\right\|_{[c,d]}< \infty.
\end{equation}
Since $\im(t-e) = \im(\bar e-t)$ for $t\in[c,d]$, we have that
\begin{eqnarray}
\left|\sum\im\left(\frac{1}{t-e}\right)\right| &=& \frac12\left| \sum \im\left(\frac{1}{t-\Delta_n(e)}\right) - \sum \im\left(\frac{1}{t-\bar e}\right) \right| \nonumber \\
{} &=& \frac12\left|\im\left(\sum\frac{\Delta_n(e)-\bar e}{(t-\Delta_n(e))(t-\bar e)}\right)\right| \leq \frac{1}{2s^2}\sum |\Delta_n(e)-\bar e|, \nonumber
\end{eqnarray}
where the sums are taken over $e\in E_n\cap\C$ and $s>0$ is such that $|t-e|\geq s$ for all $e\in E_n$ and $n\in\N$. So, (\ref{eq:limarg}) and therefore (\ref{eq:qnm}) follow from the admissibility of $\E$.

It is well-known \cite[Lem. 6.1.2]{StahlTotik} and is easily seen from the defining properties of Pad\'e approximants and the Fubini-Tonelli theorem that
\begin{equation}
\label{eq:orthpade}
\int t^jq(t)q_n(t)\frac{d\mes(t)}{v_n(t)} = 0, \quad j=0,\ldots,n-m-1,
\end{equation}
and
\begin{equation}
\label{eq:errpade}
(\ct-\Pi_n)(z) = \frac{v_n(z)}{(q_nql_{n-m})(z)}\int\frac{(q_nql_{n-m})(t)}{z-t}\frac{d\mes(t)}{v_n(t)}, \quad z\in\da,
\end{equation}
for any polynomial $_{n-m}l$ of degree at most $n-m$. Now, using decomposition (\ref{eq:qnm}) and denoting
\begin{equation}
\label{eq:nun}
d\nu_n := \frac{q_{n+m,m}q}{v_{n+m}}d\mes = \frac{q_{n+m,m}qh\hbar\hbar_{\bf x}}{v_{n+m}}d\ed,
\end{equation}
orthogonality relations (\ref{eq:orthpade}) become
\[
\int t^ju_n(t)d\nu_n(t)=0, \quad j=0,\ldots,n-1.
\]
It is also quite easy to see that the asymptotic behavior of $u_n$ is governed by Theorem \ref{thm:ortho} applied with $h_n=q_{n+m,m}q$. The orthogonality relation above also imply that
\begin{equation}
\label{eq:Rn}
R_n(z) := \int\frac{u_n(t)}{z-t}d\nu_n(t) = \frac{1}{u_n(z)}\int\frac{u_n^2(t)}{z-t}d\nu_n(t), \quad z\in\dr.
\end{equation}
Thus, putting $l_{n-m}=u_{n-m}$, we can rewrite (\ref{eq:errpade}) as
\begin{equation}
\label{eq:pp1}
\ct-\Pi_n = \frac{v_nR_{n-m}}{u_{n-m}q_{n,m}q}.
\end{equation}
Hence, we derive from (\ref{eq:nhop}) and (\ref{eq:qnm}) that
\begin{equation}
\label{eq:almostdone1}
\ct-\Pi_n = [1+o(1)] \frac{v_n\gamma_{n-m}\szf_{n-m}^2}{\sr qq_{n,m}} = [2+o(1)] \gm_{\dot\nu_{n-m}}\szf_{\dot\nu_{n-m}}^2\frac{\map^{2(n-m)}v_n}{\sr q_{n,m}q}
\end{equation}
locally uniformly in $\dr$. Therefore, we get from (\ref{eq:szegopoly}) and (\ref{eq:nun}) that
\begin{eqnarray}
\label{eq:almostdone2}
\gm_{\dot\nu_{n-m}}\szf_{\dot\nu_{n-m}}^2\frac{\map^{2(n-m)}v_n}{q_{n,m}q} &=& \gm_{\dot\mes}\szf_{\dot\mes}^2 ~ \frac{\gm_{q_{n,m}}\szf_{q_{n,m}}^2}{q_{n,m}\map^m} ~ \frac{\gm_q\szf_q^2}{q\map^m} ~ \frac{v_n\map^{2n}}{\gm_{v_n}\szf_{v_n}^2} \nonumber \\
{} &=& \frac{\gm_{\dot\mes}\szf_{\dot\mes}^2~r_n}{r_m(q_{n,m};\cdot)r} = [1+o(1)]\gm_{\dot\mes}\szf_{\dot\mes}^2 ~ \frac{r_n}{r^2}
\end{eqnarray}
locally uniformly in $\da$, where $r_n$ and $r$ are defined as in the statement of this theorem. Combining (\ref{eq:almostdone1}) with (\ref{eq:almostdone2}) we get (\ref{eq:pade1}).

Finally, observe that the boundedness of the variation of argument of $h$ was needed in order to appeal to \cite[Thm. 2.4]{BY09}. However, when the rational summand of $\ct$ is not present ($q\equiv1$), (\ref{eq:pade1}) is a consequence of Theorem \ref{thm:ortho} only and the latter does not require the boundedness of the variation of argument of $h$.
\end{proof}
\begin{proof}[Proof of Theorem \ref{thm:pade2}]
The asymptotic equality in (\ref{eq:pade2}) is exactly the one in (\ref{eq:qnm}). The fact that $u_{n-m}$ have no zeros on compact sets in $\dr$ follows since the asymptotic behavior of $u_{n-m}$ is governed by Theorem \ref{thm:ortho} with $\nu_n$ given by (\ref{eq:nun}) and all the zeros of such orthogonal polynomials approach $[c,d]$.

Let $\eta\in\z(q)$. As in the proof of Theorem \ref{thm:mer4}, we may suppose without loss of generality that $m(\eta)=m$, i.e. $\eta$ is the only zero of $q$. Using the notation of Theorem \ref{thm:ortho}, we can rewrite (\ref{eq:pp1}) as
\[
\frac{u_{n-m}^2q_{n,m}q}{\gamma_nv_n}(\ct-\Pi_n) = \frac{1}{\gamma_{n-m}}u_{n-m}R_{n-m} =: y_n,
\]
or equivalently
\begin{equation}
\label{eq:pp2}
y_n = Y_nq_{n,m} + \frac{u_{n-m}(\ct_\mes q_n-p_n)}{\gamma_{n-m} v_n}q, \quad Y_n:= \frac{u_{n-m}^2p}{\gamma_{n-m} v_n},
\end{equation}
where $\ct=\ct_\mes+p/q$. It follows from (\ref{eq:nhop}) that
\[
y_n\sr = 1+ o(1) \quad \mbox{locally uniformly in} \quad \dr.
\]
In particular, it means that sequences $\{|y_n^{(k)}(\eta)|\}$ are uniformly bounded above and away from zero for all $k\in\N$. Moreover, (\ref{eq:pp2}) yields that
\[
y_n^{(k)}(\eta) = (Y_nq_{n,m})^{(k)}(\eta), \quad k=0,\ldots,m-1.
\]
This, for instance, implies that neither of $\eta_{k,n}$, $k=1,\ldots,m$, the zeros of $q_{n,m}$,  is equal to $\eta$. Further, using (\ref{eq:nhop}), (\ref{eq:szegopoly}), and (\ref{eq:qnm}), we get that
\[
Y_n = [1/2+o(1)]\map^{2m}p/(\gm_{\dot\mes q^2}\szf_{\dot\mes q^2}^{2}r_n)
\]
uniformly in some neighborhood of $\eta$. Now, it it clear that we may proceed exactly as in the proof of Theorem \ref{thm:mer4} with the only difference that we set
\[
\left(A_{k,n}^\eta\right)^m := (1+\delta_{k,n})r_n(\eta)/\chi_n^m = - (1+\delta_{k,n}) (r_ny_n/Y_n)(\eta). 
\]
\end{proof}

\section{Numerical Experiments}
\label{sec:num}

The Hankel operator $\ho_f$ with symbol $f\in H^\infty+C(\T)$ is of finite rank if and only if $f$ is a rational function \cite[Thm. 3.11]{Partington}. In practice one can only compute with finite rank operators, due to the necessity of ordering the singular values, so a preliminary rational approximation to $f$ is needed when the latter is not rational. One way to handle this problem is to truncate the Fourier series of $f$ at some high order $N$. This provides us with a rational function $f_N$ that approximates $f$ in the Wiener norm which, in particular, dominates any $L^p$ norm on the unit circle, $p\in[1,\infty]$. It was proved in \cite{HTG90} that the best approximation operator from $H^\infty_n$ (mapping $f$ to $g_n$ according to (\ref{eq:errFunDec})) is continuous in the Wiener norm provided $(n+1)$-st singular value of the Hankel operator is simple. It was shown in \cite[Cor. 2]{BLP00} that the last assertion is satisfied for Hankel operators with symbols in some open dense subset of $H^\infty+C(\T)$, and the same technique can be used to prove that it is also the case for the particular subclass (\ref{eq:ct}). Thus, even though the simplicity of singular values cannot be asserted beforehand, it is generically true. When it prevails, one can approximates $f_N$ instead of $f$ and get a close approximation to $g_n$ when $N$ is large enough. This amounts to perform the singular value decomposition of $\ho_{f_N}$ (see \cite[Ch. 16]{Young}).

As to Pad\'e approximants, we restricted ourselves to the classical case and we constructed their denominators by solving the orthogonality relations (\ref{eq:orthpade}) with $v_n\equiv1$. Thus, finding these denominators amounts to solving a system of linear equations whose coefficients are obtained from the moments of the measure $\mu$.

The following computations were carried with MAPLE 8 software using 35 digits precision. On the figures the solid line stands for the support of the measure and circles denote the poles of the correspondent approximants. The approximated function is given by the formula
\begin{eqnarray}
\ct(z) &=& \int_{[-0.7,0]}\frac{7e^{it}}{z-t}\frac{dt}{\sqrt{(t+0.7)(0.4-t)}} + \int_{[0,0.4]}\frac{it+1}{z-t}\frac{dt}{\sqrt{(t+0.7)(0.4-t)}} \nonumber \\
{} && +\frac{1}{5!(z-0.7-0.2i)^6}. \nonumber
\end{eqnarray}

\begin{figure}[!ht]
\centering
\includegraphics[scale=.3]{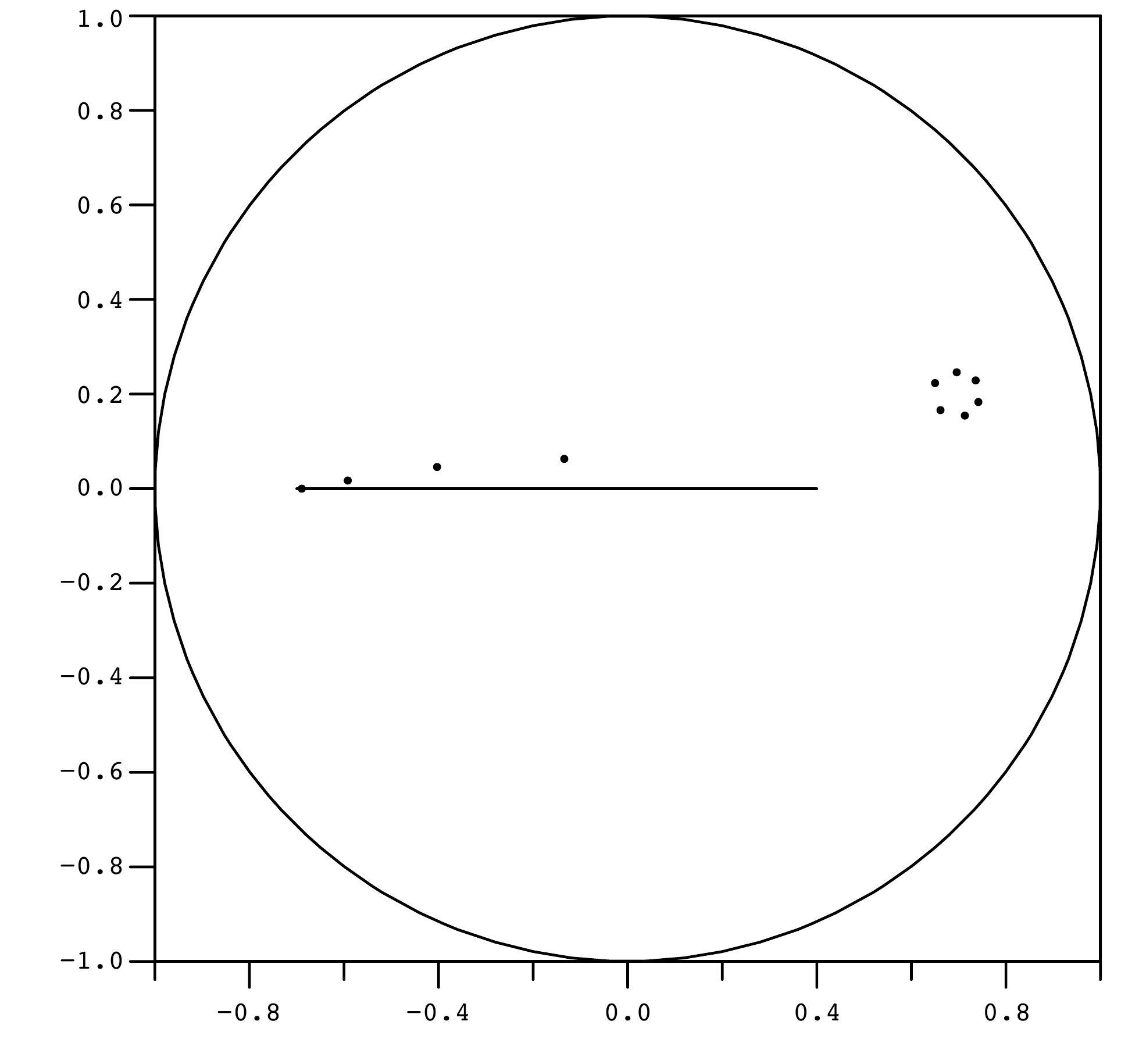}
\includegraphics[scale=.3]{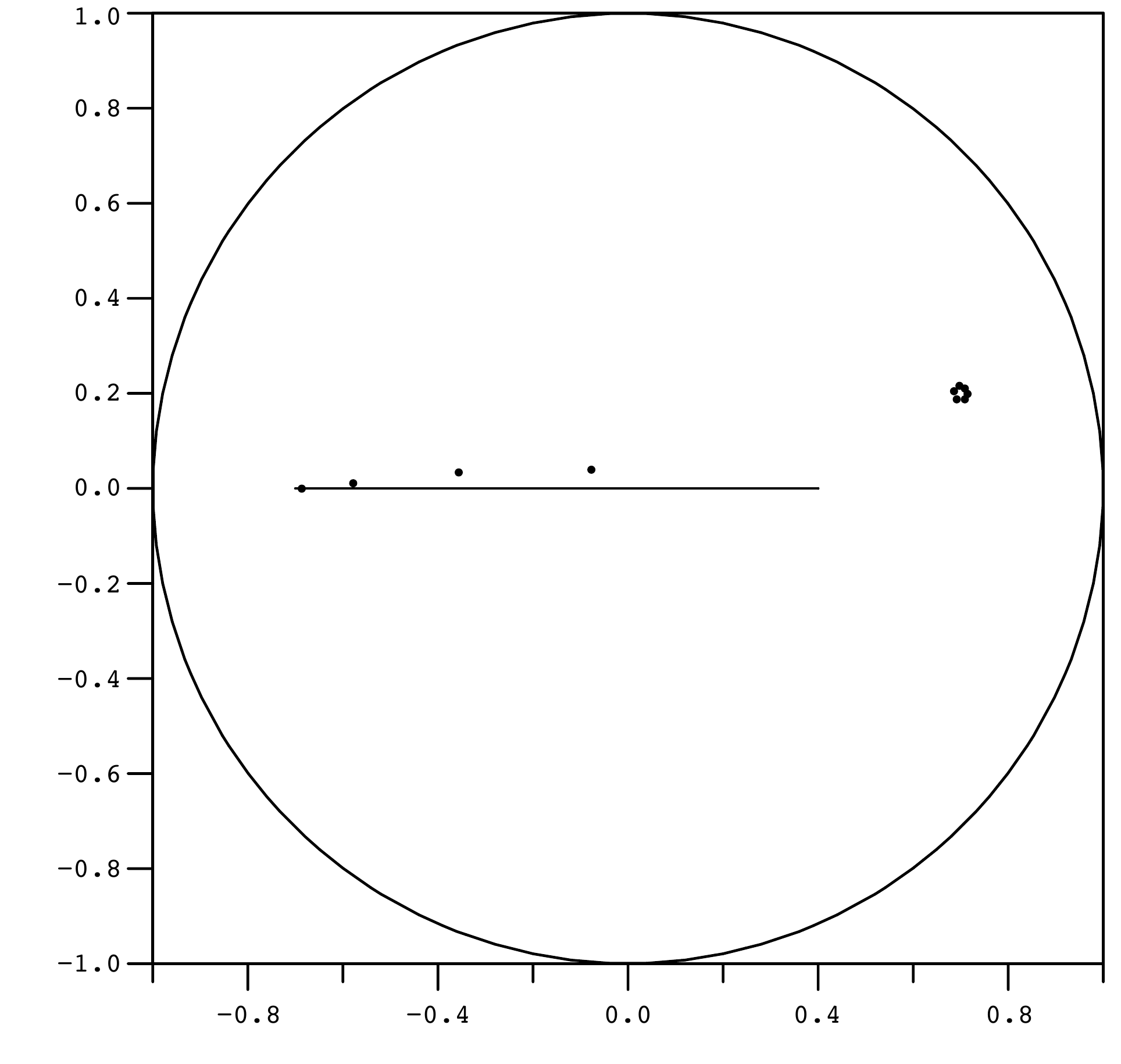}
\caption{\small Poles of Pad\'e (left) and AAK (right) approximants of degree 10.}
\end{figure}
\begin{figure}[!ht]
\centering
\includegraphics[scale=.3]{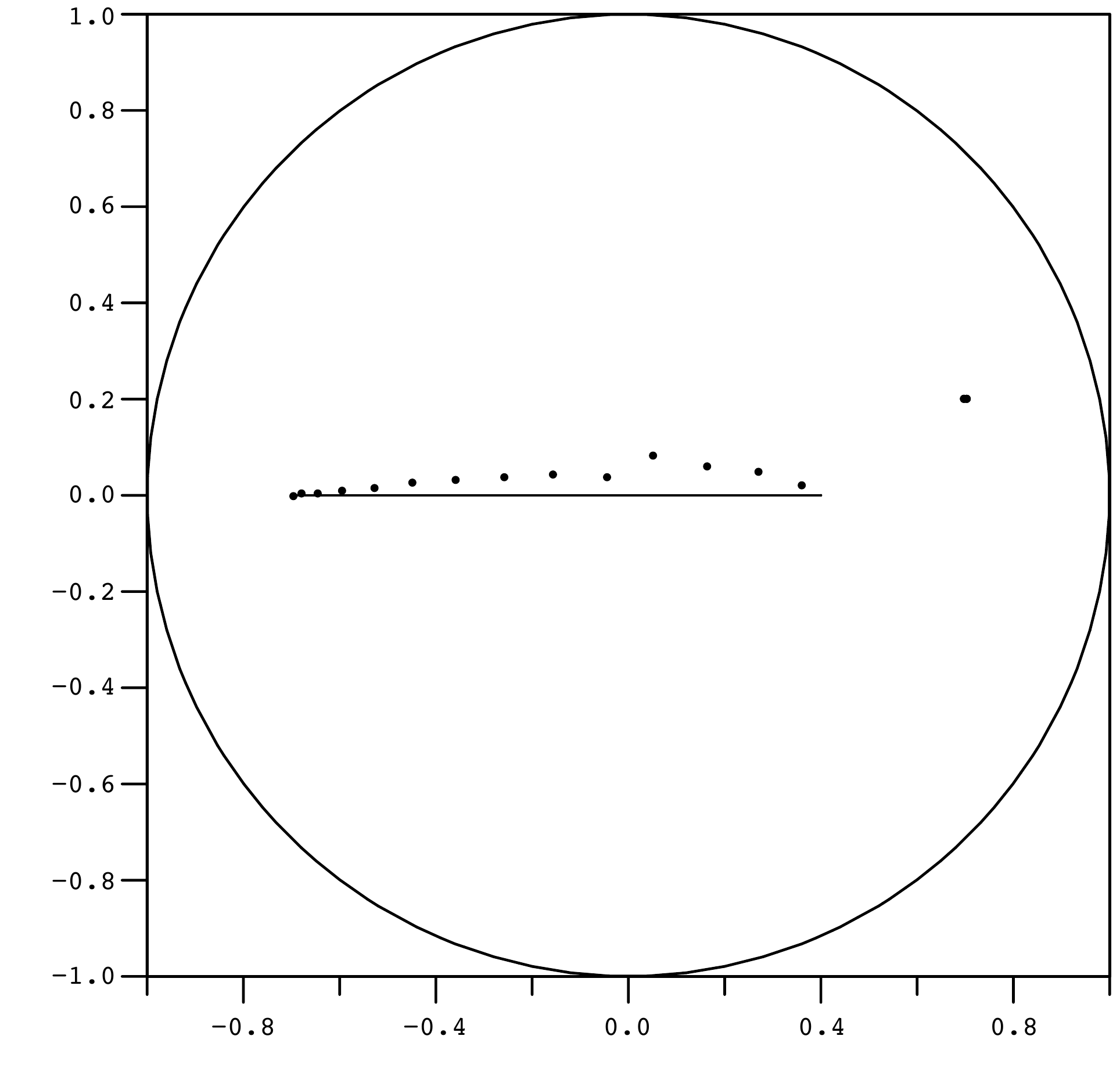}
\includegraphics[scale=.3]{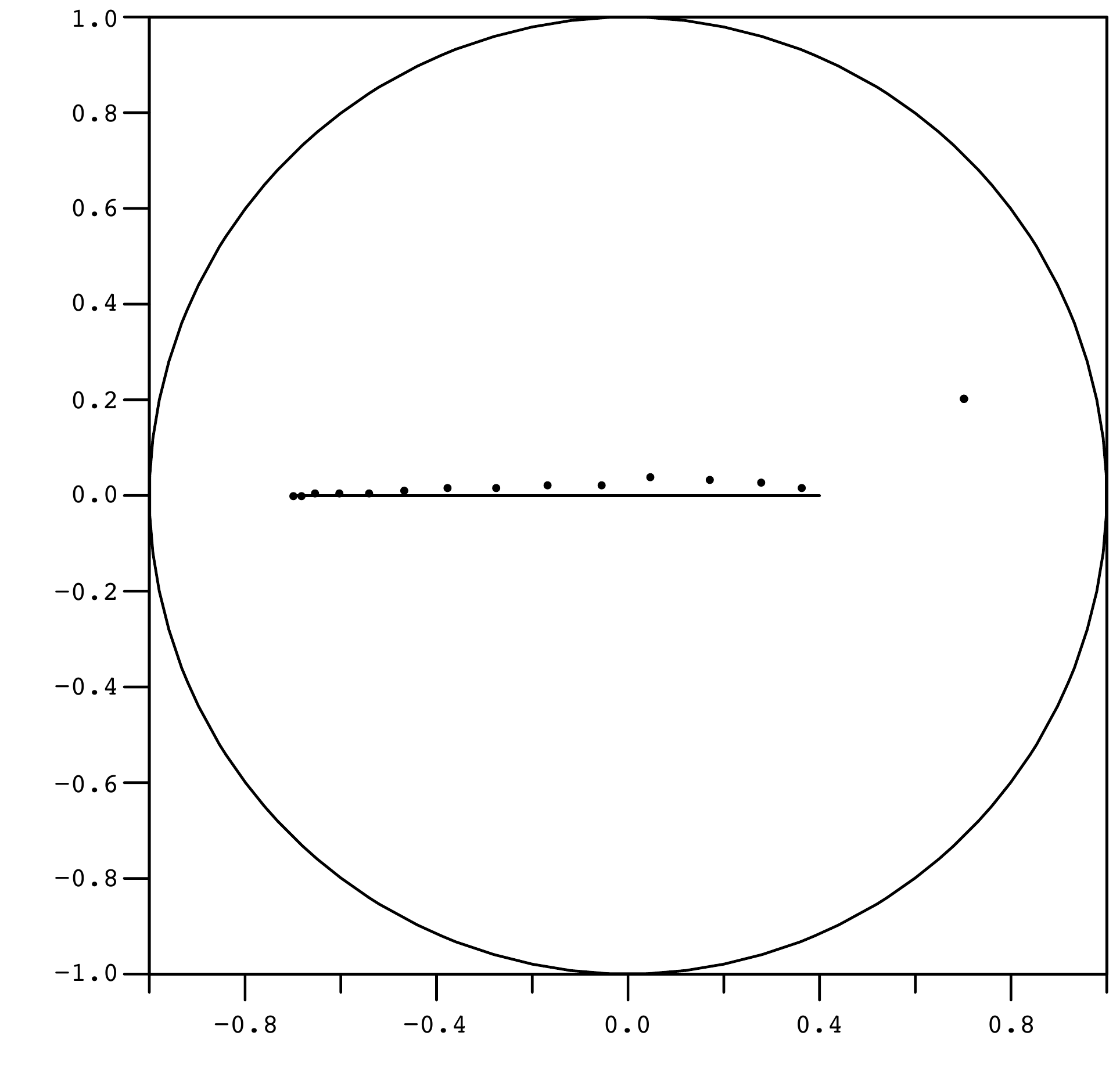}
\caption{\small Poles of Pad\'e (left) and AAK (right) approximants of degree 20.}
\end{figure}
\begin{figure}[!ht]
\centering
\includegraphics[scale=.3]{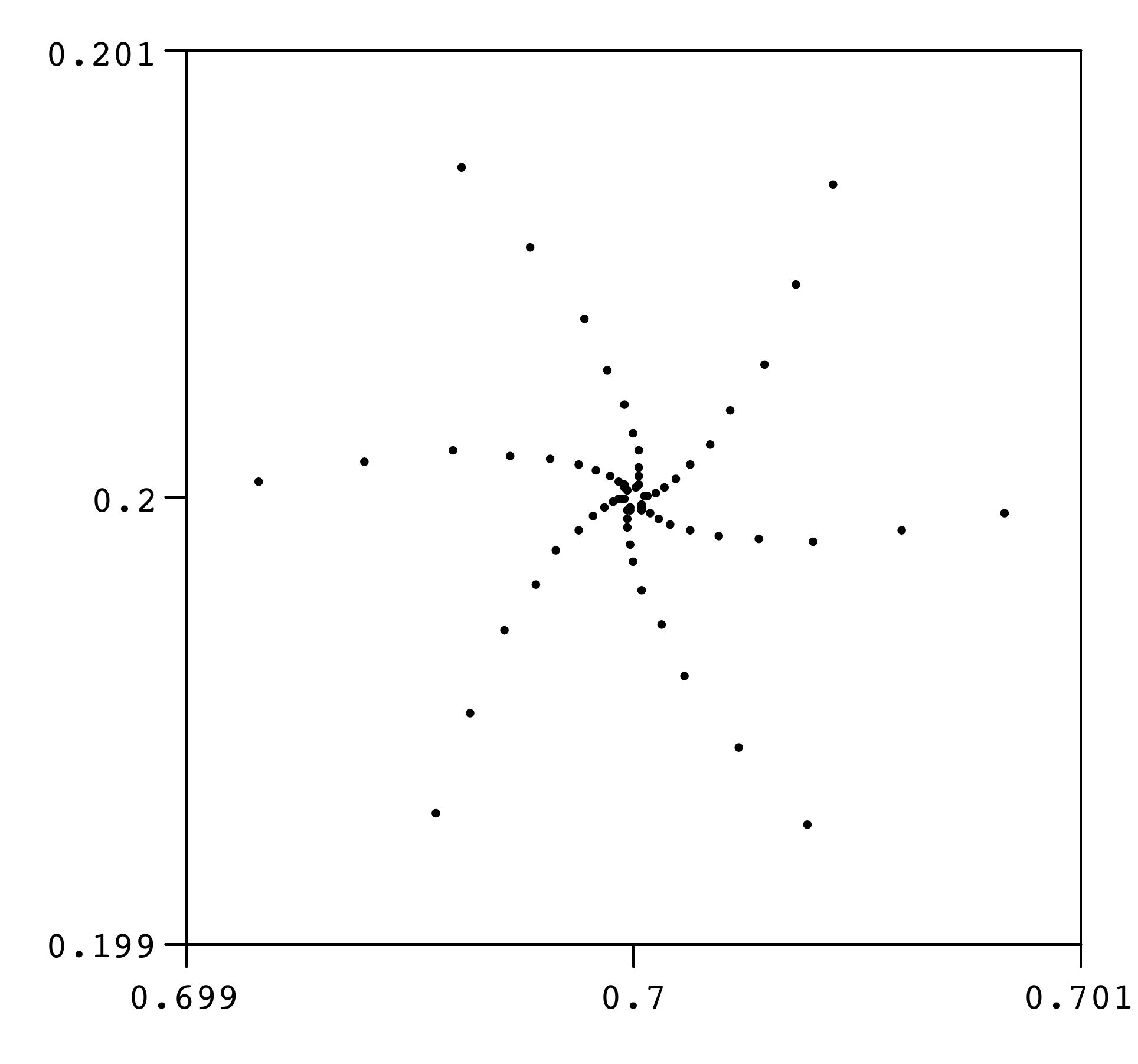}
\includegraphics[scale=.3]{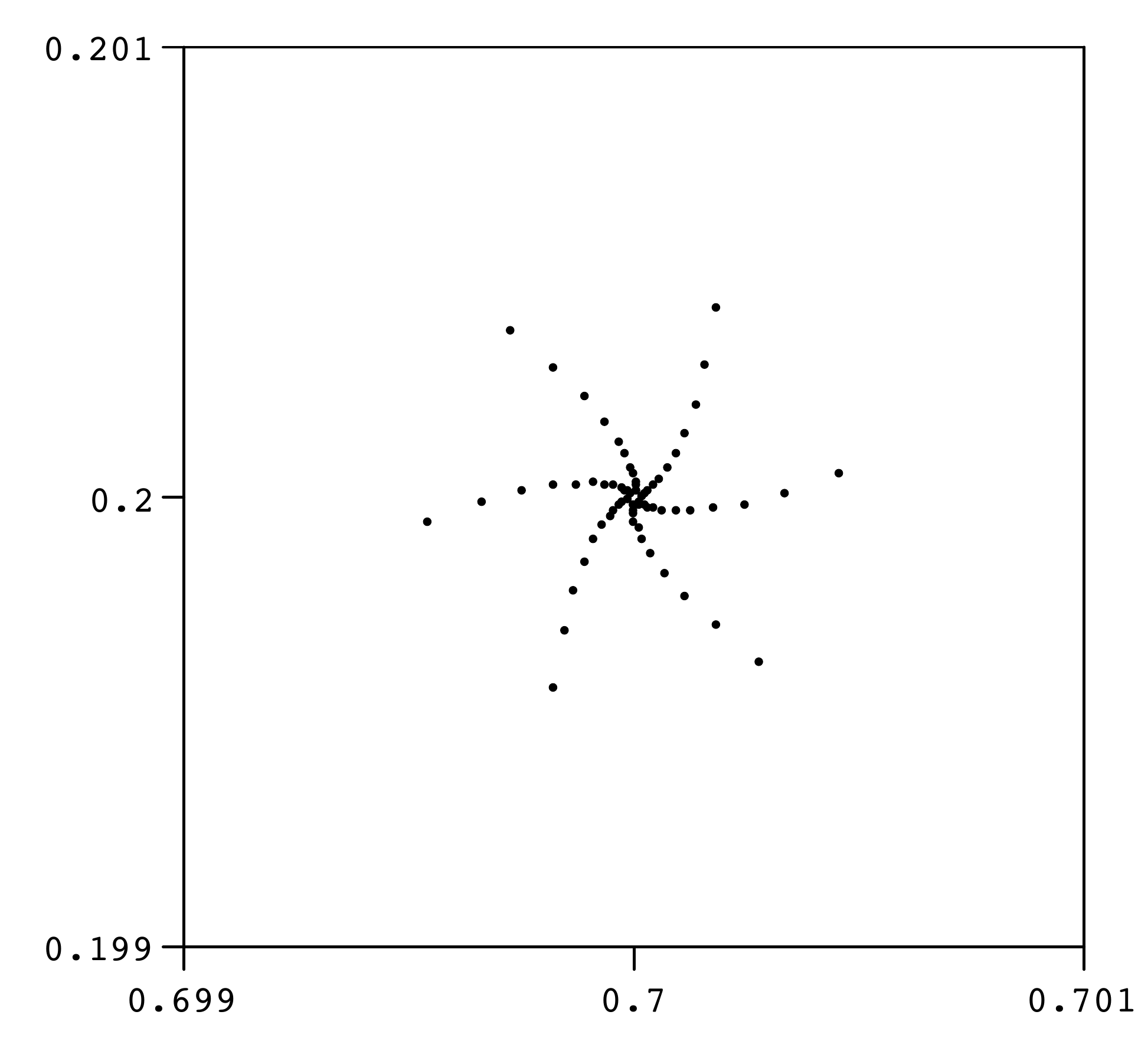}
\caption{\small Poles of Pad\'e (left) and AAK (right) approximants of degrees 21-33 lying in an neighborhood of the polar singularity.}
\end{figure}

\begin{acknowledgement}
I express my sincere gratitude to Dr. L. Baratchart for valuable discussions and comments, his reading the manuscript and suggesting this problem.
\end{acknowledgement}

\bibliographystyle{plain}
\bibliography{poles}

\end{document}